\documentclass[12pt,reqno]{amsart}
\usepackage[margin=1in]{geometry}
\usepackage{amsmath,amssymb,amsthm,graphicx,amsxtra, setspace}
\usepackage[utf8]{inputenc}
\usepackage{mathrsfs}
\usepackage{hyperref}
\usepackage{upgreek}
\usepackage{mathtools}
\usepackage[mathcal]{euscript}
\usepackage{xcolor}
\allowdisplaybreaks

\usepackage[pagewise]{lineno}

\newtheorem{theorem}{Theorem}[section]
\newtheorem{lemma}[theorem]{Lemma}
\newtheorem{proposition}[theorem]{Proposition}

\newtheorem{definition}[theorem]{Definition}

\newtheorem{remark}[theorem]{Remark}

\newtheorem{hypothesis}[theorem]{Hypothesis}

\let\originalleft\left
\let\originalright\right
\renewcommand{\left}{\mathopen{}\mathclose\bgroup\originalleft}
\renewcommand{\right}{\aftergroup\egroup\originalright}

\newcommand{\Tr}{\mathop{\mathrm{Tr}}}
\renewcommand{\d}{\/\mathrm{d}\/}

\def\w{\textbf{W}^{\varepsilon}_{{\theta}^{\varepsilon}}}

\def\t{t\wedge\tau_N}
\def\s{t\wedge\tau_N}

\def\L{\mathbb{L}}
\def\A{\mathrm{A}}

\def\F{\mathrm{F}}
\def\C{\mathrm{C}}

\def\B{\mathrm{B}}
\def\D{\mathrm{D}}

\def\Y{\mathrm{Y}}

\def\E{\mathbb{E}}
\def\X{\mathbb{X}}
\def\x{\boldsymbol{x}}
\def\k{\boldsymbol{k}}

\def\p{\boldsymbol{p}}
\def\h{\boldsymbol{h}}

\def\v{\boldsymbol{v}}
\def\w{\boldsymbol{w}}
\def\W{\mathrm{W}}
\def\G{\mathrm{G}}

\def\M{\mathrm{M}}
\def\N{\mathbb{N}}

\def\V{\mathbb{V}}
\def\wi{\widetilde}

\def\u{\mathrm{U}}
\def\P{\mathrm{P}}
\def\u{\boldsymbol{u}}
\def\H{\mathbb{H}}

\newcommand{\R}{\mathbb{R}}

\renewcommand{\d}{\/\mathrm{d}\/}

\newcommand{\Addresses}{{
		\footnote{
			\noindent \textsuperscript{1,2}Department of Mathematics, Indian Institute of Technology Roorkee-IIT Roorkee,
			Haridwar Highway, Roorkee, Uttarakhand 247667, INDIA.\par\nopagebreak
			\noindent  \textit{e-mail:} \texttt{Manil T. Mohan: maniltmohan@ma.iitr.ac.in, maniltmohan@gmail.com.}
			
			\textit{e-mail:} \texttt{Kush Kinra: kkinra@ma.iitr.ac.in.}
			
			\noindent \textsuperscript{*}Corresponding author.
			
			\textit{Key words:} Wong-Zakai approximation, stochastic convective Brinkman-Forchheimer equations, bounded domains, support theorem.
			
			Mathematics Subject Classification (2020): Primary 35R60, 76D03; Secondary 37H05

}}}

\begin{document}
	
	\title[Wong-Zakai Approximation and Support Theorem for stochastic CBF
	]{Wong-Zakai Approximation and Support Theorem for 2D and 3D stochastic convective Brinkman-Forchheimer equations
		\Addresses}
	
	\author[K. Kinra and M. T. Mohan]
	{Kush Kinra\textsuperscript{1} and Manil T. Mohan\textsuperscript{2*}}

	\maketitle
	
	\begin{abstract}
		In this work, we demonstrate the Wong-Zakai approximation results for two and three dimensional stochastic convective Brinkman-Forchheimer (SCBF) equations forced by Hilbert space valued Wiener noise on bounded domains. Even though the existence and  uniqueness of a pathwise strong solution to SCBF equations is known, the existence of a unique solution to the approximating system is not immediate from the solvability results of SCBF equations, and we prove it by using Faedo-Galerkin approximation technique and monotonicity arguments. Moreover, as an application of  the Wong-Zakai approximation, we obtain the support of the distribution of solutions to SCBF equations.  
	\end{abstract}

	\section{Introduction} \label{sec1}\setcounter{equation}{0}
	This work is devoted to the Wong-Zakai approximation of incompressible stochastic convective Brinkman-Forchheimer (CBF) equations (a non-Darcy model) which express the motion of fluid flows in a saturated porous medium (cf. \cite{PAM}). We consider the stochastic CBF equations in a bounded subset $\mathcal{O}\subset\mathbb{R}^d\ (d=2,3)$ as
	\begin{equation}\label{1}
		\left\{
		\begin{aligned}
			\d \u(t)&=[\mu \Delta\u(t)-(\u(t)\cdot\nabla)\u(t)-\alpha\u(t)-\beta|\u(t)|^{r-1}\u(t)-\nabla p(t)]\d t\\&\quad+\G(\u(t))\d \W(t), \ \ \ \text{ in } \mathcal{O}\times(0,\infty), \\ \nabla\cdot\u(t)&=0, \qquad  \qquad\qquad\ \ \ \ \ \text{ in } \ \mathcal{O}\times(0,\infty), \\
			\u(0)&=\x,\ \qquad\qquad\qquad\  \ \ x\in \mathcal{O} ,
		\end{aligned}
		\right.
	\end{equation}
	where $\u(x,t) :	\mathcal{O}\times(0,\infty)\to \mathbb{R}^d$ denotes the velocity field, $p(x,t):	\mathcal{O}\times(0,\infty)\to\R$ represents the pressure field, $\W(\cdot)$ is a cylindrical Wiener process in some separable Hilbert space defined on a complete filtered probability space $(\Omega,\mathscr{F},(\mathscr{F}_t)_{t\geq0},\mathbb{P})$ and the noise coefficient $\G(\cdot)$ satisfies certain growth and local monotonicity  conditions (see section \ref{sec2} for more details). Here the positive constants $\mu,\alpha$ and $\beta$ represent the \emph{Brinkman  (effective viscosity)}, \emph{Darcy (permeability of porous medium)} and \emph{Forchheimer coefficients}, respectively. The \emph{absorption exponent} $1\leq r<\infty$ and  $r=3$ is known as the \emph{critical exponent}. One can consider CBF equations as damped Navier-Stokes equations (NSE) due to the presence of linear and nonlinear damping  $\alpha\u+\beta|\u|^{r-1}\u$. Furthermore, the critical homogeneous CBF equations (when $r=3$) and NSE have the same scaling (Proposition 1.1, \cite{HR}) only when $\alpha=0$ but no scale invariance property for other values of $\alpha$ and $r$. Therefore it is also known as the tamed NSE (\cite{MRXZ}) or NSE modified by an absorption term (\cite{SNA}).
	
	Let us now discuss some literature available on the solvability of  CBF equations. 
	For the existence of unique weak as well as strong solutions for deterministic CBF equations on periodic and bounded domains, the readers are referred to the works \cite{SNA,FHR,HR,PAM,MTM}, etc. 
		For stochastic CBF equations \eqref{1}, the author in \cite{MTM2,MTM4} proved that there exists a pathwise unique strong (in the probabilistic sense) solution of 2D SCBF equations for any $r\geq1$ and 3D SCBF equations for any $r\geq3$ ($r>3$ for any $\mu,\beta>0$ and $r=3$ for $2\beta\mu\geq1$) forced by Gaussian and pure jump noise, respectively. The existence of a weak martingale solution for 2D and 3D SCBF equations perturbed by L\'evy noise (for $r\geq 1$)  is established in \cite{MTM3}. Similar to 3D stochastic NSE (resp. deterministic NSE), the existence of a unique pathwise strong  solution (resp. unique global strong solution) of 3D stochastic CBF (resp. deterministic CBF) equations (for $r\in[1,3)$ and $r=3$ with $2\beta\mu<1$) is still an open problem. 
	
	It is well known that one can approximate a solution to a stochastic differential equation (SDE) via several methods such as the Crank-Nicholson scheme, the Euler scheme and the Wong-Zakai approximation etc. and a good number of works have been done for each scheme (see \cite{HLN,HLN1,TMRZ,WZ} etc. and the references therein).  In \cite{WZ}, authors introduced a new idea to approximate the solution of an SDE driven by an one dimensional Brownian motion by the solution of an SDE when the Brownian motion is replaced by an appropriate smooth approximation and making a drift correction in the original SDE, which is now known as Wong-Zakai  approximation. The extension of work \cite{WZ} to the multidimensional case can be  found in \cite{WI}. Apart from the finite dimensional cases, Wong-Zakai approximation gained its attention in infinite dimensional case also. For example, in \cite{Ganguly,IPR,KT,KT1}, authors proved the convergence of Wong-Zakai approximation for stochastic (partial) differential equations in infinite dimensions. The author in \cite{KT1}  provided approximation results in finite as well as infinite dimensional SDEs and introduced two new forms of correction terms in the Wong-Zakai approximation.  For a class of 2D hydrodynamical models  like  NSE,  magneto-hydrodynamic (MHD) equations and  magnetic B\'enard problem etc., the Wong-Zakai approximation results  are established in \cite{ICAM}. The authors in \cite{HP} proved Wong-Zakai approximation for 1D parabolic nonlinear SPDEs driven by space-time white noise.  Wong-Zakai approximation  for several  fluid dynamics models and a class of stochastic partial differential equations (SPDEs) is available in \cite{BMM,LS,LQ,TMRZ,Yastrzhembskiy}, etc. and the references therein. 
	
	In \cite{TMRZ}, authors considered SPDEs with locally monotone coefficients driven by trace class noise which covers several physically relevant models such as 2D hydrodynamical type systems (stochastic NSE, stochasic MHD etc.), stochastic porous media equations, $p$-Laplace evolution equation etc. For stochastic CBF equations, we point out that  2D stochastic CBF equations with $r\in[1,3)$ comes under the functional framework of the work \cite{TMRZ} (see Remark \ref{rin1,3} below) due to the local monotonicity property satisfied by linear and nonlinear operators (Theorem \ref{LocMon}). But for $r\geq3$ in both two and three dimensions, one has to give special attention for the Wong-Zakai approximation of SCBF equations due to the fast growing nonlinearities. We also mention that the noise as well as the assumptions on the noise and correction term of the approximated system  in \eqref{1}  is taken same as that in \cite{TMRZ}. Therefore, same the examples of the noise given in \cite{TMRZ} (section 3, \cite{TMRZ}) will work for our model also. 
	
	There is an interesting advantage of Wong-Zakai approximation, that is, Wong-Zakai approximation helps us to describe the topological support of solutions of SDEs and SPDEs. The concept of topological support of solutions was first introduced by Stroock and Varadhan in \cite{SV}. In the works \cite{GP,WI}, the support theorem for SDEs has been studied. From \cite{GP,MS,Mackevicius}, it is clear that one can obtain the topological support of  solutions of SDEs and SPDEs with the help of Wong-Zakai approximation. In other words, we can also say that the support theorem is an application of Wong-Zakai approximation results, see \cite{ICAM,TMRZ,Yastrzhembskiy}, etc. Therefore we are also addressing the support of stochastic CBF equations by using the Wong-Zakai approximation results (see Section \ref{sec4}).
	
	In the literature, it has also been noticed that Wong-Zakai approximation is also used to prove the existence of random attractors for SPDEs when the diffusion term is nonlinear. In several fluids dynamics models, it is possible to convert SPDEs into a  pathwise deterministic system  when the noise is additive or linear multiplicative  and these converted systems helps us to define a continuous random dynamical system (cf. \cite{FY,KM1,PeriodicWang}, etc). To deal with other kind of noises  in the context of random attractors, Wong-Zakai approximation also plays an important role, see \cite{KM,LW,WLW} etc. and the references therein.
	
	An interesting question arises in the Wong-Zakai approximation is that what is the rate of convergence of this approximation? The works \cite{BKBAAL,IPR,NT} etc. deal with the rate of convergence of Wong-Zakai approximation for different SPDEs. Moreover, in \cite{IPR}, authors proved that the rate of convergence of Wong-Zakai approximation for SPDEs driven by Wiener processes is essentially the same as the rate of convergence of the driving processes (say $\W_n(t)$) approximating the original Wiener process $\W(t)$ under  certain conditions. The rate of convergence of Wong-Zakai approximation of stochastic CBF equations will be addressed in a future work.
	
The major aims  of this work are listed below: 
	\begin{enumerate}
		\item For $d=2,3$ with $r\geq3$ (for $d=r=3$ with $2\beta\mu\geq1$), we prove the Wong-Zakai approximation result for stochastic CBF equations \eqref{1} (Theorem \ref{Main}). For the case $d=2$ with $r\in[1,3)$, stochastic CBF equations satisfy the functional framework of the work \cite{TMRZ} (Remark \ref{RemarkB}) and hence the Wong-Zakai approximation result follows from the work \cite{TMRZ}.
		\item For $d=2,3$ with $r\geq3$ (for $d=r=3$ with $2\beta\mu\geq1$), we also obtain the topological support of the distribution to the stochastic CBF equations \eqref{1} with the help of Wong-Zakai approximation (Theorem \ref{ST}). 
	\end{enumerate} 
	
The organization of further sections of this article is as follows. In the next section, we define some functional spaces and operators (linear and nonlinear operators) needed to obtain an abstract formulation as well as the main results of this paper. We also define an approximating system and standard hypotheses for convergence of Wong-Zakai approximation (see Hypotheses \ref{HonG} and \ref{HonG1} below) in the same section. Finally, we provide the results for existence of a unique solution of stochastic CBF as well as its approximating system in the same section. Section \ref{sec3} is devoted to establish the main result of this article, that is, Wong-Zakai approximation results for stochastic CBF equations. In Section \ref{sec4}, we describe the support of solution of \eqref{1}, which is a consequence of the Wong-Zakai approximation result. In the final section, we prove the existence and uniqueness of strong solutions to the system \eqref{C_SCBF} (see section \ref{sec4} and Theorem \ref{4.1}). From the solvability result of \eqref{1} (Proposition \ref{E_U4}), we cannot directly conclude that the approximating system \eqref{WZ_SCBF} is solvable. Since system \eqref{WZ_SCBF} is a special case of the system \eqref{C_SCBF}, one can conclude the solvability results of the approximating system \eqref{WZ_SCBF} from Theorem \ref{4.1}.
	
	\section{Mathematical Formulation}\label{sec2}\setcounter{equation}{0}
	In this section, we provide the necessary function spaces needed to obtain the results of this work. Then we define some operators to set up an abstract formulation. We fix  $\mathcal{O}$ as a bounded subset of $\R^d$ with $\mathrm{C}^2$-boundary. 
	\subsection{Function spaces} 
	We define the space $$\mathscr{V}:=\{\u\in\C_0^{\infty}(\mathcal{O};\mathbb{R}^d):\nabla\cdot\u=0\},$$ where $\C_0^{\infty}(\mathcal{O};\mathbb{R}^d)$ denotes the space of all infinite times differentiable functions ($\mathbb{R}^d$-valued) with compact support in $\mathbb{R}^d$. Let $\H$ be the closure of $\mathscr{V}$ in space $\L^2(\mathcal{O})=\mathrm{L}^2(\mathcal{O};\R^d)$ with the norm  $\|\u\|_{\H}^2:=\int_{\mathcal{O}}|\u(x)|^2\d x,
	$ and inner product $(\u,\v)=\int_{\mathcal{O}}\u(x)\cdot\v(x)\d x,$ for all $\u,\v\in\L^2(\mathcal{O})$, respectively. Let $\V$ be the closure of $\mathscr{V}$ in space $\H_0^1(\mathcal{O})=\mathrm{H}_0^1(\mathcal{O};\R^d)$ with the norm $ \|\u\|_{\V}^2:=\int_{\mathcal{O}}|\nabla\u(x)|^2\d x,
	$ and the inner product $(\!(\u,\v)\!)=(\nabla\u,\nabla\v)=\int_{\mathcal{O}}\nabla\u(x)\cdot\nabla\v(x)\d x,$ for all $\u,\v\in\V$, respectively. Let $\widetilde{\L}^{p}$ be the closure of  $\mathscr{V}$  in space  $\L^p(\mathcal{O})=\mathrm{L}^p(\mathcal{O};\R^d),$ for $p\in(2,\infty)$, with the norm $\|\u\|_{\widetilde{\L}^p}^p=\int_{\mathcal{O}}|\u(x)|^p\d x.$ Let $\langle \cdot,\cdot\rangle $ represent the induced duality between the spaces $\V$  and its dual $\V'$ as well as $\widetilde{\L}^p$ and its dual $\widetilde{\L}^{p'}$, where $\frac{1}{p}+\frac{1}{p'}=1$. Note that $\H$ can be identified with its dual $\H'$. We endow the space $\V\cap\widetilde{\L}^{p}$ with the norm $\|\u\|_{\V}+\|\u\|_{\widetilde{\L}^{p}},$ for $\u\in\V\cap\widetilde{\L}^p$ and its dual $\V'+\widetilde{\L}^{p'}$ with the norm $$\inf\left\{\max\left(\|\v_1\|_{\V'},\|\v_2\|_{\widetilde{\L}^{p'}}\right):\v=\v_1+\v_2, \ \v_1\in\V', \ \v_2\in\widetilde{\L}^{p'}\right\}.$$ Moreover, we have the continuous embedding $\V\cap\widetilde{\L}^p\hookrightarrow\V\hookrightarrow\H\equiv\H'\hookrightarrow\V'\hookrightarrow\V'+\widetilde{\L}^{p'}$. One can define equivalent norms on $\V\cap\widetilde\L^{p}$ and $\V'+\widetilde\L^{\frac{p}{p-1}}$ as 
	\begin{align*}
		\|\u\|_{\V\cap\widetilde\L^{p}}=\left(\|\u\|_{\V}^2+\|\u\|_{\widetilde\L^{p}}^2\right)^{\frac{1}{2}}\ \text{ and } \ \|\u\|_{\V'+\widetilde\L^{\frac{p}{p-1}}}=\inf_{\u=\v+\w}\left(\|\v\|_{\V'}^2+\|\w\|_{\widetilde\L^{\frac{p}{p-1}}}^2\right)^{\frac{1}{2}}.
	\end{align*}
 Further, we use the notation $\H^2(\mathcal{O}):=\mathrm{H}^2(\mathcal{O};\R^d)$ for the second order Sobolev spaces.
 	\subsection{Inequalities}
 The following inequalities are frequently used  in the paper. 
 \begin{lemma}[H\"older's inequality]\label{Holder}
 	Assume that $\frac{1}{p}+\frac{1}{p'}=1$ with $1\leq p,p'\leq\infty$, $\u_1\in\L^{p}(\mathcal{O})$ and $\u_2\in\L^{p'}(\mathcal{O})$, then we get 
 	\begin{align*}
 		\|\u_1\u_2\|_{\L^1(\mathcal{O})}\leq\|\u_1\|_{\L^{p}(\mathcal{O})}\|\u_2\|_{\L^{p'}(\mathcal{O})}. 
 	\end{align*}
 \end{lemma}
 \begin{lemma}[Interpolation inequality]\label{Interpolation}
 	Assume $1\leq s_1\leq s\leq s_2\leq \infty$, $\theta\in(0,1)$ such that $\frac{1}{s}=\frac{a}{s_1}+\frac{1-a}{s_2}$ and $\u\in\L^{s_1}(\mathcal{O})\cap\L^{s_2}(\mathcal{O})$, then we have 
 	\begin{align*}
 		\|\u\|_{\L^s(\mathcal{O})}\leq\|\u\|_{\L^{s_1}(\mathcal{O})}^{a}\|\u\|_{\L^{s_2}(\mathcal{O})}^{1-a}. 
 	\end{align*}
 \end{lemma}
 \begin{lemma}[Young's inequality]\label{Young}
 	For all $a,b,\varepsilon>0$ and for all $1<p,p'<\infty$ with $\frac{1}{p}+\frac{1}{p'}=1$, we obtain 
 	\begin{align*}
 		ab\leq\frac{\varepsilon}{p}a^p+\frac{1}{p'\varepsilon^{p'/p}}b^{p'}.
 	\end{align*}
 \end{lemma}
	\subsection{Projection and linear operators }
	Let $\mathcal{P}_p: \L^p(\mathcal{O}) \to\widetilde{\L}^p$ be the Helmholtz-Hodge (or Leray) projection  (cf.  \cite{DFHM}). For $p=2$, $\mathcal{P}:=\mathcal{P}_2$ becomes an orthogonal projection  and for $2<p<\infty$, it is a bounded linear operator. Since $\mathcal{O}$ has $\mathrm{C}^2$-boundary, $\mathcal{P}$ maps $\H^1(\mathcal{O})$ into itself (see Remark 1.6, \cite{Temam}).

	We define the Stokes operator
	\begin{equation*}
		\A\u:=-\mathcal{P}\Delta\u,\;\u\in\D(\A):=\V\cap\H^{2}(\mathcal{O}).
	\end{equation*}
	\subsection{Bilinear operator}
	Let us define the \emph{trilinear form} $b(\cdot,\cdot,\cdot):\V\times\V\times\V\to\R$ by $$b(\u,\v,\w)=\int_{\mathcal{O}}(\u(x)\cdot\nabla)\v(x)\cdot\w(x)\d x=\sum_{i,j=1}^d\int_{\mathcal{O}}\u_i(x)\frac{\partial \v_j(x)}{\partial x_i}\w_j(x)\d x.$$ If $\u, \v$ are such that the linear map $b(\u, \v, \cdot) $ is continuous on $\V$, the corresponding element of $\V'$ is denoted by $\B(\u, \v)$. We also denote $\B(\u) = \B(\u, \u)=\mathcal{P}[(\u\cdot\nabla)\u]$.
	An integration by parts gives 
	\begin{align}\label{b0}
		b(\u,\v,\w) &=  -b(\u,\w,\v),\ \text{ for all }\ \u,\v,\w\in \V.
	\end{align}
Consequently, $b(\u,\v,\v)= 0,$ for all $\u,\v \in\V$.
	\begin{remark}
	Note that $\langle\B(\v,\u-\v),\u-\v\rangle=0$ and it implies that
	\begin{align}\label{441}
		\langle \B(\u)-\B(\v),\u-\v\rangle =\langle\B(\u-\v,\u),\u-\v\rangle=-\langle\B(\u-\v,\u-\v),\v\rangle.
	\end{align}
\end{remark}
	\begin{remark}
		Let us recall the following estimates on trilinear form $b(\cdot,\cdot,\cdot)$ from \cite{Temam1} (see Chapter 2, \cite{Temam1}). For all   $\u, \v, \w\in \V$
		\begin{align}\label{b1}
			|b(\u,\v,\w)|&\leq C
			\begin{cases}
				\|\u\|^{\frac{1}{2}}_{\H}\|\u\|^{\frac{1}{2}}_{\V}\|\v\|_{\V}\|\w\|^{\frac{1}{2}}_{\H}\|\w\|^{\frac{1}{2}}_{\V},\ \ \ 	\text{ for } d=2,\\
				\|\u\|^{\frac{1}{4}}_{\H}\|\u\|^{\frac{3}{4}}_{\V}\|\v\|_{\V}\|\w\|^{\frac{1}{4}}_{\H}\|\w\|^{\frac{3}{4}}_{\V}, \ \ \ \text{ for } d=3.
			\end{cases}
		\end{align}
	\end{remark}
\begin{remark}\label{RemarkB}
	It can also be seen that for $r> 3$, $\B$ maps $\widetilde{\L}^{r+1}$ into $\V'$ and using Lemma \ref{Interpolation}, we obtain  
	\begin{align}\label{212}
		\left|\langle \B(\u),\v\rangle \right|&=\left|b(\u,\v,\u)\right|\leq \|\u\|_{\widetilde{\L}^{\frac{2(r+1)}{r-1}}}\|\u\|_{\widetilde{\L}^{r+1}}\|\v\|_{\V}\leq\|\u\|_{\widetilde{\L}^{r+1}}^{\frac{r+1}{r-1}}\|\u\|_{\H}^{\frac{r-3}{r-1}}\|\v\|_{\V},\\
		\|\B(\u)\|^2_{\V^{'}}&\leq\|\u\|_{\widetilde{\L}^{r+1}}^{\frac{2(r+1)}{r-1}}\|\u\|_{\H}^{\frac{2(r-3)}{r-1}}\leq C\bigg(\|\u\|_{\widetilde{\L}^{r+1}}^{r+1}+\|\u\|_{\H}^{2}\bigg)\leq C\bigg(\|\u\|_{\widetilde{\L}^{r+1}}^{r+1}+\|\u\|_{\V}^{2}\bigg).\nonumber
	\end{align}
For $r=3$, we get
\begin{align*}
	\|\B(\u)\|^2_{\V^{'}}\leq\|\u\|_{\widetilde{\L}^{4}}^{4}.
\end{align*}
\end{remark}
	\subsection{Nonlinear operator}
	Consider the nonlinear operator $\mathcal{C}(\u):=\mathcal{P}(|\u|^{r-1}\u),$ which implies that $\langle\mathcal{C}(\u),\u\rangle =\|\u\|_{\widetilde{\L}^{r+1}}^{r+1}$. Also, the map $\mathcal{C}(\cdot):\V\cap\widetilde{\L}^{r+1}\to\V'+\widetilde{\L}^{\frac{r+1}{r}}$. For all $\u\in\V\cap\wi\L^{r+1}$, the map is Gateaux differentiable with Gateaux derivative 
	\begin{align}\label{29}
	\mathcal{C}'(\u)\v&=\left\{\begin{array}{cl}\mathcal{P}(\v),&\text{ for }r=1,\\ \left\{\begin{array}{cc}\mathcal{P}(|\u|^{r-1}\v)+(r-1)\mathcal{P}\left(\frac{\u}{|\u|^{3-r}}(\u\cdot\v)\right),&\text{ if }\u\neq \mathbf{0},\\\mathbf{0},&\text{ if }\u=\mathbf{0},\end{array}\right.&\text{ for } 1<r<3,\\ \mathcal{P}(|\u|^{r-1}\v)+(r-1)\mathcal{P}(\u|\u|^{r-3}(\u\cdot\v)), &\text{ for }r\geq 3,\end{array}\right.
	\end{align}
	for all $\v\in\V\cap\widetilde{\L}^{r+1}$. Moreover, for any $r\in [1, \infty)$ and $\u, \v \in \V\cap\widetilde{\L}^{r+1}$, we obtain (see subsection 2.4, \cite{MTM2})
	\begin{align}\label{MO_c}
		\langle\mathcal{C}(\u)-\mathcal{C}(\v),\u-\v\rangle\geq\frac{1}{2}\||\u|^{\frac{r-1}{2}}(\u-\v)\|_{\H}^2+\frac{1}{2}\||\v|^{\frac{r-1}{2}}(\u-\v)\|_{\H}^2 \geq 0,
	\end{align}
	and 
	\begin{align}\label{a215}
		\|\u-\v\|_{\wi\L^{r+1}}^{r+1}\leq 2^{r-2}\||\u|^{\frac{r-1}{2}}(\u-\v)\|_{\H}^2+2^{r-2}\||\v|^{\frac{r-1}{2}}(\u-\v)\|_{\H}^2,
	\end{align}
	for $r\geq 1$ (replace $2^{r-2}$ with $1,$ for $1\leq r\leq 2$).
	
		\begin{theorem}[\cite{MTM2}]\label{LocMon}
		Let $d=2$ with $ r\in[1,3]$, $d=2,3$ with $ r> 3$, $d=r=3$ with $2\beta\mu\geq1$ and $\u_1, \u_2 \in \V.$ Then, for the operator $\M(\u) = \mu\A\u +\B(\u)+\alpha\u+\beta\mathcal{C}(\u),$ we have 
		\begin{align}\label{fe2_1}
			\langle\M(\u_1)-\M(\u_2),\u_1-\u_2\rangle+ \frac{27}{32\mu ^3}\|\u_2\|^4_{\widetilde{\L}^4}\|\u_1-\u_2\|_{\H}^2&\geq 0, \text{ for } d=2 \text{ and } r\in[1,3],
		\end{align}
		\begin{align}\label{fe2_2}
			\langle\M(\u_1)-\M(\u_2),\u_1-\u_2\rangle+ \eta\|\u_1-\u_2\|_{\H}^2&\geq 0, \text{ for } d=2,3 \text{ and } r>3,
		\end{align}
		where $\eta=\frac{r-3}{2\mu(r-1)}\left[\frac{2}{\beta\mu (r-1)}\right]^{\frac{2}{r-3}}$ and 
		\begin{align}\label{fe2_3}
			\langle\M(\u_1)-\M(\u_2),\u_1-\u_2\rangle \geq 0, \text{ for } d=r=3 \text{ with } 2\beta\mu\geq1.
		\end{align} 
	\end{theorem}	
	
	\subsection{Abstract formulation and solvability results}\label{AF}
	In this subsection, we describe an abstract formulation and solution of the system \eqref{1}. Taking orthogonal projection $\mathcal{P}$ to the first equation in \eqref{1}, we obtain
	\begin{equation}\label{SCBF}
		\left\{
		\begin{aligned}
			\d \u(t)+[\mu \A\u(t)+\B(\u(t))+\alpha\u(t)+\beta\mathcal{C}(\u(t))]\d t&= \G(\u(t))\d\W(t), \ \ t>0, \\
			\u(0)&=\x,
		\end{aligned}
		\right.
	\end{equation}
where $\W(\cdot)$ is a cylindrical Wiener process in a separable Hilbert space $\left(\mathrm{K},\left\langle\cdot,\cdot\right\rangle_{\mathrm{K}}\right)$ defined on a complete filtered probability space $(\Omega,\mathscr{F},(\mathscr{F}_t)_{t\geq0},\mathbb{P})$. Furthermore,  $$\G:\H\times\Omega\to\mathcal{L}_2(\mathrm{K};\H),$$ where $\left(\mathcal{L}_2(\mathrm{K};\H),\|\cdot\|_{\mathcal{L}_2}\right)$ is the space of all Hilbert-Schmidt operators from $\mathrm{K}$ to $\H$. For the solvability of stochastic system \eqref{SCBF}, we need following assumptions on the noise coefficient $\G(\cdot)$.
	\begin{hypothesis}\label{HonG}
		The noise coefficient $\G(\cdot)$ satisfies the following conditions:
		\begin{itemize}
			\item[(H.1)] \emph{(Growth condition)} There exists a positive constant $L_1$ such that for all $\v\in\H$,$$\|\G(\v)\|^2_{\mathcal{L}_2}\leq L_1(1+\|\v\|_{\H}^2).$$
			\item[(H.2)] \emph{(Local monotonicity condition)} There exists a positive measurable function $\rho:\V\to[0,\infty)$ such that for all $\v_1,\v_2\in\H$ and $\v\in\V$,$$\|\G(\v_1)-\G(\v_2)\|^2_{\mathcal{L}_2}\leq \rho(\v_2)\|\v_1-\v_2\|_{\H}^2,$$
			where \begin{align*}
				\rho(\v)\leq L_1(1+\|\v\|^{2}_{\V})(1+\|\v\|^{\gamma}_{\H}).
			\end{align*}
		\end{itemize}
	\end{hypothesis}

Let us first provide the global solvability results of the system \eqref{SCBF}.
\begin{definition}
	Let $p>\gamma+2$, $\x\in \mathrm{L}^{p}(\Omega, \mathscr{F}_0,\mathbb{P};\H)$. Then, an $\H$-valued $(\mathscr{F}_t)$-adapted stochastic process $\{\u(t)\}_{t\in[0,T]}$ is called a \emph{strong solution} (in probabilistic sense) of \eqref{SCBF} on $[0,T]$ with initial data $\x$ if $\u\in\mathrm{L}^p(\Omega;\mathrm{L}^{\infty}(0,T;\H))\cap\mathrm{L}^2(\Omega;\mathrm{L}^2(0,T;\V))\cap\mathrm{L}^{r+1}(\Omega;\mathrm{L}^{r+1}(0,T;\widetilde{\L}^{r+1})),$ and satisfies, for every $t>0$ and $\phi\in\V\cap\widetilde{\L}^{r+1}$,
	\begin{align*}
		(\u(t),\phi)
		&+ \int_{0}^{t} \langle\mu\A\u(\zeta)+ \B(\u(\zeta))+\alpha\u(\zeta)+\beta\mathcal{C}(\u(\zeta)),\phi\rangle\d \zeta \\&= (\x,\phi)+\int_{0}^{t} (\phi, \G(\u(\zeta))\d\W(\zeta)),
	\end{align*}$\mathbb{P}$-a.s., and $\u$ has  a modification with paths in  $\u\in\mathrm{C}([0,T];\H)\cap\mathrm{L}^2(0,T;\V)\cap\mathrm{L}^{r+1}(0,T;\widetilde{\L}^{r+1}), $ $\mathbb{P}$-a.s. 

A strong solution $\u(\cdot)$ to the system \eqref{SCBF} is called a
	\emph{pathwise  unique strong solution} if
	$\widetilde{\u}(\cdot)$ is an another strong
	solution, then $$\mathbb{P}\big\{\omega\in\Omega:\u(t)=\widetilde{\u}(t),\ \text{ for all }\ t\in[0,T]\big\}=1.$$ 
\end{definition}
Under the Hypothesis \ref{HonG} on $\G$, the existence of a unique pathwise strong solution to the system \eqref{SCBF} is given by the following result, which can be obtained from the works \cite{MTM2,MTM1,MTM3}, etc.
\begin{proposition}\label{E_U1}
	For $d=2$ with $r\in[1,3]$, $d=2,3$ with $r\in(3,\infty)$ and $d=r=3$ with $2\beta\mu\geq1$, assume that all the conditions of Hypothesis \ref{HonG} are satisfied. Then, for every $T>0$, $p>\gamma+2$, for every $\x\in \mathrm{L}^{p}(\Omega, \mathscr{F}_0,\mathbb{P};\H)$, system \eqref{SCBF} has a pathwise unique strong solution $\u\in\mathrm{L}^{p}(\Omega; \mathrm{L}^{\infty}(0,T;\H))\cap\mathrm{L}^2(\Omega; \mathrm{L}^2(0,T;\V))\cap\mathrm{L}^{r+1}(\Omega;\mathrm{L}^{r+1}(0,T;\widetilde{\L}^{r+1}))$, that is,
	\begin{align}\label{E_U2}
		\mathbb{E}\left[\sup_{t\in[0,T]}\|\u(t)\|^{p}_{\H}+\int_{0}^{T}\|\u(t)\|^2_{\V}\d t+\int_{0}^{T}\|\u(t)\|^{r+1}_{\wi\L^{r+1}}\d t\right]<\infty,
	\end{align}
with paths in $\C([0,T];\H)\cap\mathrm{L}^2(0,T;\V)\cap\mathrm{L}^{r+1}(0,T;\wi\L^{r+1})$, $\mathbb{P}$-a.s.
\end{proposition}
\begin{proof}
	See the works \cite{MTM2,MTM1,MTM3}.
\end{proof}
\begin{remark}\label{rin1,3}
	In \cite{TMRZ}, authors demonstrate the Wong-Zakai approximation and support theorem for SPDEs with locally monotone coefficients. We show that for $d=2$ with $r\in[1,3)$ our stochastic system \eqref{SCBF} comes under the framework of the work \cite{TMRZ}.  Let us show that for $d=2$ with $r\in[1,3)$, the system \eqref{SCBF} satisfies  \emph{Assumption 1} of the work \cite{TMRZ}. Since, assumptions on the noise coefficient $\mathrm{G}(\cdot)$ (Hypothesis \ref{HonG}) is same as that  in \cite{TMRZ}, we only need to check assumptions on $\mathrm{M}(\cdot)$. The conditions (H1) Hemicontinuity, (H2) Local monotonicity and (H3) Coercivity of Assumption 1 in \cite{TMRZ} are easy to check (cf. \cite{MTM2,MTM1}). Finally, we show that $\mathrm{M}(\cdot)$ satisfies   condition (H4) of Assumption 1 in \cite{TMRZ}. For $\varepsilon\in(0,1)$, using \eqref{b1} (for $d=2$), Sobolev's embedding (for $d=2$), Lemmas \ref{Holder} and \ref{Young}, we have
	\begin{align*}
		|\big\langle\M(\u),\v\big\rangle|&=|\big\langle\mu\A\u+\B(\u)+\alpha\u+\beta\mathcal{C}(\u),\v\big\rangle|\\&\leq \mu\|\u\|_{\V}\|\v\|_{\V}+C\|\u\|_{\H}\|\u\|_{\V}\|\v\|_{\V}+\alpha\|\u\|_{\H}\|\v\|_{\H}+\beta\|\u\|^r_{\wi\L^{r(1+\varepsilon)}}\|\v\|_{\wi\L^{\frac{1+\varepsilon}{\varepsilon}}}\\&\leq C\big[\|\u\|_{\V}+\|\u\|_{\H}\|\u\|_{\V}+\|\u\|^{r}_{\wi\L^{r(1+\varepsilon)}}\big]\|\v\|_{\V},
	\end{align*}
	where $\varepsilon\in(0,1)$. Now since $r<3$, we can choose $\varepsilon$ small enough such that $r<1+\frac{2}{1+\varepsilon}$, so that 
	$$q:=\frac{2}{2-(r-1)(1+\varepsilon)}\in(1,\infty).$$
	Let $q':=\frac{q}{q-1}=\frac{2}{(r-1)(1+\varepsilon)}$ and $\lambda\in(0,1)$. Applying interpolation inequality and choosing $\lambda:=\frac{1}{r}$ (assuming WLOG that $r>1$), we obtain
	\begin{align*}
		\|\u\|^{r}_{\wi\L^{r(1+\varepsilon)}}&\leq\|\u\|^{\lambda r}_{\wi\L^{\lambda qr(1+\varepsilon)}}\|\u\|^{(1-\lambda)r}_{\wi\L^{(1-\lambda)q'r(1+\varepsilon)}}=\|\u\|_{\wi\L^{q(1+\varepsilon)}}\|\u\|^{r-1}_{\H}\\&\leq C\|\u\|_{\V}\|\u\|^{r-1}_{\H},
	\end{align*}
	which implies that
	\begin{align*}
		\|\M(\u)\|^2_{\V'}&\leq C\left[\|\u\|^2_{\V}+\|\u\|^2_{\H}\|\u\|^2_{\V}+\|\u\|^2_{\V}\|\u\|^{2(r-1)}_{\H}\right]\\&\leq C(1+\|\u\|^2_{\V})(1+\|\u\|^{2\max\{r-1,1\}}_{\H}).
	\end{align*}
	Hence condition  (H4)  of Assumption 1 in \cite{TMRZ} is satisfied.
\end{remark}
\begin{remark}
In the light of Remark \ref{rin1,3}, we will consider the case $d=2$ with $r=3$, $d=2,3$ with $r\in(3,\infty)$ and $d=r=3$ with $2\beta\mu\geq1$ only in the rest of the paper. 
\end{remark}
\subsection{An approximation of the system \eqref{SCBF}}
In this subsection, we present an approximating system for the stochastic CBF equations \eqref{SCBF}. For that, we first define an adapted finite-dimensional approximation of the Wiener process $\W(\cdot)$. Indeed, for a fixed orthonormal basis $\{e_k\}_{k\in\N}$ of $\mathrm{K}$ and a sequence $\{w_k\}_{k\in\N}$ of independent Brownian motions defined on a complete filtered probability space $(\Omega, \mathscr{F}, (\mathscr{F}_t)_{t\in \R}, \mathbb{P})$ such that $\W(t)$ can be written in the following form (\cite{DZ1})
\begin{align*}
	\W(t)=\sum_{k=1}^{\infty}w_k(t) e_k,  \ \ \ t\in[0,T].
\end{align*}
For $n\in\N$, we set $\sigma=\frac{T}{2^n}$ and define
\begin{align}\label{AS1}
	\dot{\W}^n(t) = \sum_{k=1}^{n}\frac{1}{\sigma}\left\{w_k\big(\Big\lfloor\frac{t}{\sigma}\Big\rfloor\sigma\big)-w_k\big((\Big\lfloor\frac{t}{\sigma}\Big\rfloor-1)\sigma\big)\right\}e_k=:\sum_{k=1}^{n} \dot{w}_k^n(t)e_k, \ \ t\in[0,T],
\end{align}
where $\lfloor s\rfloor$ denotes the greatest integer function for $s\in[0,T]$. Also, we set
\begin{align*}
	w_k(t)=\begin{cases}
		0 \ \ &\text{ for } t\leq0,\\
		w_k(T) \ \ &\text{ for } t\geq T,
	\end{cases}
\end{align*}
therefore $\dot{w}_k(t)=0$ for $t>T$. Then $\dot{w}^n_k(t), k=1,\ldots,n$ are $\mathscr{F}_t$-adapted and consequently $\dot{\W}^n(t)$.

For $k=1,\ldots,n$, let $\G_k:\H\to\H$ be defined by $\G_k(\u)=\G(\u)e_k, \u\in\H.$ We assume that for each $k$, $\G_k$ is Fr\'echet  differentiable with its derivative denoted by $\D\G_k:\H\to\mathcal{L}(\H;\H)$. Then we define the map
\begin{align}\label{AS2}
	\widetilde{\Tr}_n:\H\to\H \text{ such that } \widetilde{\Tr}_n(\u)=\sum_{k=1}^{n}\D\G_k(\u)\G_k(\u), \ \ \u\in\H.
\end{align}
Let us consider the approximating equations as 
	\begin{equation}\label{WZ_SCBF}
	\left\{
	\begin{aligned}
		\d \u^n(t)+[\mu \A\u^n(t)+\B(\u^n(t))&+\alpha\u^n(t)+\beta\mathcal{C}(\u^n(t))]\d t\\&= \G(\u^n(t))\dot{\W}^n(t)\d t-\frac{1}{2}\widetilde{\Tr}_n(\u^n(t))\d t, \\
		\u^n(0)&=\x,
	\end{aligned}
	\right.
\end{equation}
where $\dot{\W}^n$ and $\widetilde{\Tr}_n$ are given in \eqref{AS1} and \eqref{AS2}, respectively. The following assumption is needed to obtain the existence and uniqueness of the approximating system \eqref{WZ_SCBF}, which is similar to the conditions for Wong-Zakai approximation in the literature (see \cite{ICAM,TMRZ} etc. and references therein).
\begin{hypothesis}\label{HonG1}
	For each $k\in\N$, the map $\G_k$ is twice Fr\'echet differentiable with its second Fr\'echet derivative denoted by $\D^2\G_k:\H\to\mathcal{L}(\H;\mathcal{L}(\H,\H))\cong\mathcal{L}(\H\times\H;\H),$ and satisfies:
	\begin{itemize}
		\item[(H$'$.1)] For any $M>0$, there exists a positive constant $C(M)$ such that 
		\begin{align*}
			\sup_{k\in\N}\sup_{\|\v\|_{\H}\leq M} \left\{\|\D\G_k(\v)\|_{\mathcal{L}(\H;\H)}\lor\|\G_k(\v)\|_{\H}\lor\|\D^2\G_k(\v)\|_{\mathcal{L}(\H\times\H;\H)}\right\}\leq C(M),
		\end{align*}
	\begin{align*}
		\D\G^{*}_k|_{\V\cap\widetilde{\L}^{r+1}}:\V\cap\widetilde{\L}^{r+1}\to\V\cap\widetilde{\L}^{r+1},
	\end{align*}
\begin{align*}
	  \sup_{k\in\N}\sup_{\|\v\|_{\H}\leq M}\|\D\G_k(\v)^{*}\u\|_{\widetilde{\L}^{r+1}}&\leq C(M)\|\u\|_{\widetilde{\L}^{r+1}}, \ \ \u\in\widetilde{\L}^{r+1},\\
	  \sup_{k\in\N}\sup_{\|\v\|_{\H}\leq M}\|\D\G_k(\v)^{*}\u\|_{\V}&\leq C(M)\|\u\|_{\V}, \ \ \u\in\V,
\end{align*}
	and for $m\in\N$,
	\begin{align*}
		\lim_{m\to\infty} \sup_{\|\v\|_{\H}\leq M}\|\G(\v)-\G(\v)\circ\Pi_m\|_{\mathcal{L}_2}=0,
	\end{align*}
where $\Pi_m$ represents the orthogonal projection onto $\mathrm{K}_m:=\{e_1,\cdots,e_m\}$ in $\mathrm{K}$, that is, $\Pi_m\x=\sum_{j=1}^{m}\langle\x,e_j\rangle_{\mathrm{K}}e_j, \x\in\mathrm{K},$  and $\D\G_k(\cdot)^{*}$ denotes the dual operator of $\D\G_k(\cdot)$.
\item[(H$'$.2)] There exists a constant $L_2>0$ such that  for every $n\in\N$ and $\v,\v_1,\v_2\in\H$
\begin{align*}
	\|\widetilde{\Tr}_n(\v)\|^2_{\H}\leq L_2(1+\|\v\|^2_{\H}),
\end{align*}
and 
\begin{align*}
	(\widetilde{\Tr}_n(\v_2)-\widetilde{\Tr}_n(\v_1),\v_1-\v_2)\leq \rho(\v_2)\|\v_1-\v_2\|^2_{\H},
\end{align*}
where $\rho$ is the same as given in Hypothesis \ref{HonG}.
	\end{itemize}
\end{hypothesis}
For the examples which satisfies Hypotheses \ref{HonG} and \ref{HonG1}, we refer the readers to \cite{TMRZ} (see section 3, \cite{TMRZ}). Under Hypotheses \ref{HonG} and \ref{HonG1}, we can obtain the following solvability result for the system \eqref{WZ_SCBF} (see Theorem \ref{4.1} and Section \ref{sec5} below).
\begin{proposition}\label{E_U4}
	Assume that all the conditions of Hypotheses \ref{HonG} and \ref{HonG1} are satisfied. Then, for every $T>0$, $p>\gamma+2$, for all $\x\in \mathrm{L}^{p}(\Omega, \mathscr{F}_0,\mathbb{P};\H)$, the system \eqref{WZ_SCBF} has a pathwise unique strong solution  $\u^n\in\mathrm{L}^{p}(\Omega; \mathrm{L}^{\infty}(0,T;\H))\cap\mathrm{L}^2(\Omega; \mathrm{L}^2(0,T;\V))\cap\mathrm{L}^{r+1}(\Omega;\mathrm{L}^{r+1}(0,T;\widetilde{\L}^{r+1}))$, that is,
	\begin{align}\label{E_U5}
		\sup_{n\geq1}\mathbb{E}\left[\sup_{t\in[0,T]}\|\u^n(t)\|^{p}_{\H}+\int_{0}^{T}\|\u^n(t)\|^2_{\V}\d t+\int_{0}^{T}\|\u^n(t)\|^{r+1}_{\wi\L^{r+1}}\d t\right]<\infty,
	\end{align} 
with a modification having  paths in $\C([0,T];\H)\cap\mathrm{L}^2(0,T;\V)\cap\mathrm{L}^{r+1}(0,T;\wi\L^{r+1})$, $\mathbb{P}$-a.s.
\end{proposition}
\section{Wong-Zakai approximation for stochastic CBF equations}\label{sec3}\setcounter{equation}{0}
In this section, we prove our main result of this work, that is, Wong-Zakai approximation of the stochastic CBF system \eqref{SCBF}. Let us first recall the following Lemma which was proved in \cite{ICAM}.
\begin{lemma}[Lemma 2.1, \cite{ICAM}]
	Let $T>0$. Then there exists a constant $\delta_0>0$ such that for every $\delta>\frac{\delta_0}{\sqrt{T}}$, $t\in[0,T]$,
	\begin{align}
		\lim_{n\to\infty}\mathbb{P}\left(\sup_{1\leq k\leq n}\sup_{s\leq t}\left|\dot{w}^n_k(s)\right|>\delta n^{\frac{1}{2}}2^{\frac{n}{2}}\right)&=0,\label{W1}\\
		\lim_{n\to\infty}\mathbb{P}\left(\sup_{s\leq t}\|\dot{\W}^n(s)\|_{\mathrm{K}}>\delta n2^{\frac{n}{2}}\right)&=0.\label{W2}
	\end{align}
\end{lemma}
Let us define stopping times for $M\geq0$, $n\in\N$, $\delta>\frac{\delta_0}{\sqrt{T}}$:
\begin{align*}
	\tau^{(1)}_{M}&:=\inf_{t\geq 0}\left\{t:\|\u(t)\|_{\H}+\int_{0}^{t}\|\u(\zeta)\|^2_{\V}\d\zeta+\int_{0}^{t}\|\u(\zeta)\|^{r+1}_{\wi\L^{r+1}}\d\zeta>M\right\}\land T,\\	\tau^{(2)}_{n,M}&:=\inf_{t\geq 0}\left\{t:\|\u^n(t)\|_{\H}+\int_{0}^{t}\|\u^n(\zeta)\|^2_{\V}\d\zeta+\int_{0}^{t}\|\u^n(\zeta)\|^{r+1}_{\wi\L^{r+1}}\d\zeta>M\right\}\land T,\\
	\tau^{(3)}_{n}&:=\inf_{t\geq 0}\bigg\{t:\big[\sup_{s\in[0,t]}\sup_{1\leq k\leq n}\left|\dot{w}^n_k(s)\right|\big]\lor\big[n^{-\frac{1}{2}}\sup_{s\in[0,t]}\|\dot{\W}^n(s)\|_{\mathrm{K}}\big]>\delta n^{\frac{1}{2}} 2^{\frac{n}{2}}\bigg\}\land T,
\end{align*}
and
\begin{align}\label{Tau1}
	\tau_{n,M}:=\tau^{(1)}_{M}\land\tau^{(2)}_{n,M}\land\tau^{(3)}_{n}.
\end{align}
From \eqref{E_U2}, \eqref{E_U5} and \eqref{W1}-\eqref{W2}, we infer that
\begin{align*}
	\lim_{M\to\infty}\mathbb{P}\left(\tau^{(1)}_{M}=T\right)=	\lim_{n\to\infty}\mathbb{P}\left(\tau^{(3)}_{n}=T\right)=1 \text{ and } \\
		\lim_{M\to\infty}\mathbb{P}\left(\tau^{(2)}_{n,M}=T\right)=1, \ \ \text{ uniformly for }n\in\N.
\end{align*}
\begin{remark}
 In order to deal with the difference of equations \eqref{SCBF} and \eqref{WZ_SCBF}, it is observed that the integral $\int_{0}^{\cdot} \G(\u^n(s))\dot{\W}^n(s)\d s$ cannot be considered as stochastic integral directly. Instead, we use following identity (see Remark 2.7, \cite{TMRZ})
 \begin{align}
 	\int_{0}^{t} \G\Big(\u^n\big((\Big\lfloor\frac{s}{\sigma}\Big\rfloor-1)\sigma\big)\Big)\dot{\W}^n(s)\d s=\int_{0}^{t}\bigg(\frac{1}{\sigma}\int\limits_{\lceil\frac{s}{\sigma}\rceil\sigma}^{(\lceil\frac{s}{\sigma}\rceil+1)\sigma}1_{\{\xi\leq t\}}\d\xi\bigg)\G\Big(\u^n\big(\Big\lfloor\frac{s}{\sigma}\Big\rfloor\sigma\big)\Big)\circ\Pi_n\d\W(s),
 \end{align}
to compare with the corresponding diffusion term $\int_{0}^{t} \G(\u(s))\d\W(s)$, where $\lceil s\rceil$ denotes the smallest integer function for $s\in[0,T]$.
\end{remark}
Now we state our main result of this work.
\begin{theorem}\label{Main}
	Assume that all the conditions of Hypotheses \ref{HonG} and \ref{HonG1} are satisfied, $p>\gamma+2$ and $\x\in \mathrm{L}^{p}(\Omega, \mathscr{F}_0,\mathbb{P};\H)$. Let $\u$ and $\u^n$ be the solutions to the systems \eqref{SCBF} and \eqref{WZ_SCBF}, respectively, with same initial data $\x$. Then
	\begin{align}\label{Main2}
		\lim_{n\to\infty}\mathbb{E}\left[\sup_{t\in[0,T]}\|\u(t)-\u^n(t)\|^2_{\H}\right]=0.
	\end{align}
\end{theorem}
\begin{proof}
	In order to prove \eqref{Main2}, it is enough to show that for sufficiently large $M>0$,
	\begin{align}\label{MP1}
		\lim_{n\to\infty}\mathbb{E}\bigg[\sup_{t\in[0,\tau_{n,M}]}\|\u(t)-\u^n(t)\|^2_{\H}\bigg]=0,
	\end{align}
where $\tau_{n,M}$ is given by \eqref{Tau1}. Set $\Omega_{n,M}:=\{\omega\in\Omega:\tau_{n,M}=T\}$ for $M>0$. It follows from \eqref{E_U2}, \eqref{E_U5} and \eqref{Tau1} that for any $\varepsilon>0$, there exists a constant $M_{\varepsilon}$ (sufficiently large and independent of $n$), still denoting by $M$, such that
\begin{align}\label{MP2}	\mathbb{E}\bigg[\sup_{t\in[0,T]}\chi_{\Omega^{c}_{n,M}}\|\u(t)-\u^n(t)\|^2_{\H}\bigg]\leq\mathbb{P}\left(\Omega^c_{n,M}\right)^{\frac{p-2}{2}}\mathbb{E}\bigg[\sup_{t\in[0,T]}\|\u(t)-\u^n(t)\|^p_{\H}\bigg]^{\frac{2}{p}}\leq \frac{\varepsilon}{2}.
\end{align}
 In fact, \eqref{MP2} implies that we only need to prove \eqref{MP1} to obtain \eqref{Main2}. For simplicity, we denote $\tau_{n,M}$ by $\tau_n$. 
 The proof is divided into following four steps.
 \vskip 2mm
 \noindent
 \textbf{Step 1:} According to the definition of $\tau_n$, for some fixed constant $\delta$ with $\delta>\frac{\delta_0}{\sqrt{T}}$, we can find a constant $C(M)$ such that for all $t\in[0,\tau_n]$ and $k=1,\ldots,n$,
 \begin{equation}\label{MP3}
 	\left\{
 \begin{aligned}
 	\|\u(t)\|_{\H}+\|\u^n(t)\|_{\H}&\leq C(M),\\
 	\int_{0}^{t}\|\u(\zeta)\|^2_{\V}\d \zeta + 	\int_{0}^{t}\|\u^n(\zeta)\|^2_{\V}\d \zeta&\leq C(M),\\
 	 	\int_{0}^{t}\|\u(\zeta)\|^{r+1}_{\wi\L^{r+1}}\d \zeta + 	\int_{0}^{t}\|\u^n(\zeta)\|^{r+1}_{\wi\L^{r+1}}\d \zeta&\leq C(M),\  \\
 	 	\left|\dot{w}^n_k(t)\right|+n^{-\frac{1}{2}}\|\dot{\W}^n(t)\|_{\mathrm{K}}&\leq2\delta n^{\frac{1}{2}} 2^{\frac{n}{2}}.
 \end{aligned}
	\right.
\end{equation}
Now applying It\^o's formula (cf. \cite{MTM2,MTM1}) to $\|\u^n(\cdot)-\u(\cdot)\|_{\H}^2$, we get
\begin{align}\label{MP5}
	&\|\u^n(t)-\u(t)\|^2_{\H}+2\mu\int_{0}^{t}\|\u^n(\zeta)-\u(\zeta)\|^2_{\V}\d\zeta+2\alpha\int_{0}^{t}\|\u^n(\zeta)-\u(\zeta)\|^2_{\H}\d\zeta\nonumber\\&=-2\int_{0}^{t}\left\langle\B\big(\u^n(\zeta)\big)-\B\big(\u(\zeta)\big),\u^n(\zeta)-\u(\zeta)\right\rangle\d\zeta\nonumber\\&\quad-2\beta\int_{0}^{t}\left\langle\mathcal{C}\big(\u^n(\zeta)\big)-\mathcal{C}\big(\u(\zeta)\big),\u^n(\zeta)-\u(\zeta)\right\rangle\d\zeta\nonumber\\&\quad+\int_{0}^{t}\bigg\|\bigg(\frac{1}{\sigma}\int\limits_{\lceil\frac{\zeta}{\sigma}\rceil\sigma}^{(\lceil\frac{\zeta}{\sigma}\rceil+1)\sigma}1_{\{\xi\leq t\}}\d\xi\bigg)\G\big(\u^n(\Big\lfloor\frac{\zeta}{\sigma}\Big\rfloor\sigma)\big)\Pi_n-\G\big(\u(\zeta)\big)\bigg\|^2_{\mathcal{L}_2}\d\zeta\nonumber\\&\quad+2\int_{0}^{t}\bigg(\u^n(\zeta)-\u(\zeta),\bigg[\bigg(\frac{1}{\sigma}\int\limits_{\lceil\frac{\zeta}{\sigma}\rceil\sigma}^{(\lceil\frac{\zeta}{\sigma}\rceil+1)\sigma}1_{\{\xi\leq t\}}\d\xi\bigg)\G\big(\u^n(\Big\lfloor\frac{\zeta}{\sigma}\Big\rfloor\sigma)\big)\Pi_n-\G\big(\u(\zeta)\big)\bigg]\d\W(\zeta)\bigg)\nonumber\\&\quad+2\int_{0}^{t}\bigg(\bigg[\G\big(\u^n(\zeta)\big)-\G\Big(\u^n\big((\Big\lfloor\frac{\zeta}{\sigma}\Big\rfloor-1)\sigma\big)\Big)\bigg]\dot{\W}^n(\zeta)-\frac{1}{2}\widetilde{\Tr}\big(\u^n(\zeta)\big),\ \u^n(\zeta)-\u(\zeta)\bigg)\d\zeta\nonumber\\&=:J_1(n,t)+J_2(n,t)+J_3(n,t)+J_4(n,t)+J_5(n,t).
\end{align}
\textbf{Step 2:} \textit{Claim:}There exists a constant $C(T,M)$ such that
\begin{align}
	\mathbb{E}\bigg[\int_{0}^{\tau_n}\|\u(\zeta)-\u(\Big\lfloor\frac{\zeta}{\sigma}\Big\rfloor\sigma)\|^2_{\H}\d\zeta\bigg]&\leq C(T,M)2^{-\frac{3}{4}n},\label{MP6}\\	\mathbb{E}\bigg[\int_{0}^{\tau_n}\|\u^n(\zeta)-\u^n(\Big\lfloor\frac{\zeta}{\sigma}\Big\rfloor\sigma)\|^2_{\H}\d\zeta\bigg]&\leq C(T,M)2^{-\frac{3}{4}n},\label{MP7}\\
	\mathbb{E}\bigg[\int_{0}^{\tau_n}\|\u(\zeta)-\u\big((\Big\lfloor\frac{\zeta}{\sigma}\Big\rfloor-1)\sigma\big)\|^2_{\H}\d\zeta\bigg]&\leq C(T,M)2^{-\frac{3}{4}n},\label{MP6'}\\
	\mathbb{E}\bigg[\int_{0}^{\tau_n}\|\u^n(\zeta)-\u^n\big((\Big\lfloor\frac{\zeta}{\sigma}\Big\rfloor-1)\sigma\big)\|^2_{\H}\d\zeta\bigg]&\leq C(T,M)2^{-\frac{3}{4}n},\label{MP6''}\\	\mathbb{E}\bigg[\int_{0}^{\tau_n}\|\u(\zeta)-\u(\Big\lceil\frac{\zeta}{\sigma}\Big\rceil\sigma)\|^2_{\H}\d\zeta\bigg]&\leq C(T,M)2^{-\frac{3}{4}n},\label{MP7'}\\	\mathbb{E}\bigg[\int_{0}^{\tau_n}\|\u(\zeta)-\u(\Big\lceil\frac{\zeta}{\sigma}\Big\rceil\sigma)\|^2_{\H}\d\zeta\bigg]&\leq C(T,M)2^{-\frac{3}{4}n}.\label{MP7''}
\end{align}
\vskip2mm
\noindent
\textit{Proof of \eqref{MP6}}: Applying It\^o's formula to $\|\u(\cdot)-\u(\lfloor\frac{\xi}{\sigma}\rfloor\sigma)\|_{\H}^2$, integrating with respect to $\xi$ and taking expectation, we get for $t\in(0,\tau_n]$,
\begin{align}\label{MP8}
	&\mathbb{E}\left[\int_{0}^{t}\|\u(\xi)-\u(\Big\lfloor\frac{\xi}{\sigma}\Big\rfloor\sigma)\|^2_{\H}\d\xi\right]\nonumber\\&=2\mathbb{E}\left[\int_{0}^{t}\int_{\lfloor\frac{\xi}{\sigma}\rfloor\sigma}^{\xi}\big\langle\mu\A\u(\zeta)+\B(\u(\zeta))+\alpha\u(\zeta)+\mathcal{C}(\u(\zeta)), \u(\zeta)-\u(\Big\lfloor\frac{\xi}{\sigma}\Big\rfloor\sigma)\big\rangle\d\zeta\d\xi\right]\nonumber\\&\quad+\mathbb{E}\left[\int_{0}^{t}\int_{\lfloor\frac{\xi}{\sigma}\rfloor\sigma}^{\xi}\|\G\big(\u(\zeta)\big)\|^2_{\mathcal{L}_2}\d\zeta\d\xi\right]\nonumber\\&\quad+2\mathbb{E}\left[\int_{0}^{t}\int_{\lfloor\frac{\xi}{\sigma}\rfloor\sigma}^{\xi}\big(\u(\zeta)-\u(\Big\lfloor\frac{\xi}{\sigma}\Big\rfloor\sigma),\G\big(\u(\zeta)\big)\d\W(\zeta)\big)\d\xi\right]\nonumber\\&=:I_1(n)+I_2(n)+I_3(n).
\end{align}
\vskip2mm
\noindent
\textit{Estimate for $I_1(n)$:} Using Lemmas \ref{Holder} and \ref{Young}, \eqref{b0}, Remark \ref{RemarkB}, stochastic Fubini's theorem and \eqref{E_U2}, we obtain
\begin{align}\label{MP10}
	\left|I_1(n)\right|&\leq C \mathbb{E}\bigg[\int_{0}^{\tau_n}\int_{\lfloor\frac{\xi}{\sigma}\rfloor\sigma}^{\xi}\bigg\{\|\u(\zeta)\|^2_{\H}+\|\u(\zeta)\|^2_{\V}+\|\u(\zeta)\|^{r+1}_{\wi\L^{r+1}}+\|\u(\zeta)-\u(\Big\lfloor\frac{\xi}{\sigma}\Big\rfloor\sigma)\|^2_{\H}\nonumber\\&\qquad\qquad+\|\u(\zeta)-\u(\Big\lfloor\frac{\xi}{\sigma}\Big\rfloor\sigma)\|^2_{\V}+\|\u(\zeta)-\u(\Big\lfloor\frac{\xi}{\sigma}\Big\rfloor\sigma)\|^{r+1}_{\wi\L^{r+1}}\bigg\}\d\zeta\d\xi\bigg]\nonumber\\&\leq C\sigma \mathbb{E}\bigg[\int_{0}^{T}\bigg\{\|\u(\xi)\|^2_{\H}+\|\u(\xi)\|^2_{\V}+\|\u(\xi)\|^{r+1}_{\wi\L^{r+1}}+\|\u(\xi)-\u(\Big\lfloor\frac{\xi}{\sigma}\Big\rfloor\sigma)\|^2_{\H}\nonumber\\&\qquad\qquad+\|\u(\xi)-\u(\Big\lfloor\frac{\xi}{\sigma}\Big\rfloor\sigma)\|^2_{\V}+\|\u(\xi)-\u(\Big\lfloor\frac{\xi}{\sigma}\Big\rfloor\sigma)\|^{r+1}_{\wi\L^{r+1}}\bigg\}\d\xi\bigg]\nonumber\\&\leq C(T,M)2^{-n}.
\end{align}
\vskip2mm
\noindent
\textit{Estimate for $I_2(n)$:} From condition (H.1) of Hypothesis \ref{HonG} and \eqref{MP3}, it is immediate that 
\begin{align}\label{MP11}
	\left|I_2(n)\right|\leq C(T,M)2^{-n}.
\end{align}
\vskip2mm
\noindent
\textit{Estimate for $I_3(n)$:} We infer from Burkholder-Davies-Gundy's (BDG) inequality, condition (H.1) of Hypothesis \ref{HonG} and \eqref{MP3} that 
\begin{align}
	I_3(n)&\leq2\mathbb{E}\left[\sup_{t\in[0,\tau_{n}]}\int_{0}^{t}\int_{\lfloor\frac{\xi}{\sigma}\rfloor\sigma}^{\xi}\big\langle\u(\zeta)-\u(\Big\lfloor\frac{\xi}{\sigma}\Big\rfloor\sigma),\G\big(\u(\zeta)\big)\d\W(\zeta)\big\rangle\d\xi\right]\nonumber\\&\leq C \mathbb{E}\left[\int_{0}^{\tau_{n}}\int_{\lfloor\frac{\xi}{\sigma}\rfloor\sigma}^{\xi}\|\G\big(\u(\zeta)\big)\|^2_{\mathcal{L}_2}\|\u(\zeta)-\u(\Big\lfloor\frac{\xi}{\sigma}\Big\rfloor\sigma)\|^2_{\H}\d\zeta\d\xi\right]^{\frac{1}{2}}\label{MP12}\\&\leq C(T,M)2^{-\frac{n}{2}}.\label{MP13}
\end{align}
From \eqref{MP8}-\eqref{MP13}, we obtain
\begin{align}\label{MP14}
	\mathbb{E}\bigg[\int_{0}^{\tau_n}\|\u(\xi)-\u(\Big\lfloor\frac{\xi}{\sigma}\Big\rfloor\sigma)\|^2_{\H}\d\xi\bigg]\leq C(T,M)2^{-\frac{n}{2}}.
\end{align}
Again, applying stochastic Fubini's theorem in \eqref{MP12}, using condition (H.1) of Hypothesis \ref{HonG} and \eqref{MP14}, respectively, we obtain
\begin{align}\label{MP15}
	I_3(n)&\leq C\mathbb{E}\left[\sigma\int_{0}^{\tau_{n}}\|\G\big(\u(\zeta)\big)\|^2_{\mathcal{L}_2}\|\u(\zeta)-\u(\Big\lfloor\frac{\zeta}{\sigma}\Big\rfloor\sigma)\|^2_{\H}\d\zeta\right]^{\frac{1}{2}}\nonumber\\&\leq C\mathbb{E}\left[\sigma\int_{0}^{\tau_{n}}(1+\|\u(\zeta)\|^2_{\H})\|\u(\zeta)-\u(\Big\lfloor\frac{\zeta}{\sigma}\Big\rfloor\sigma)\|^2_{\H}\d\zeta\right]^{\frac{1}{2}}\nonumber\\&\leq C(T,M)2^{-\frac{3}{4}n}.
\end{align}
Hence, combining \eqref{MP8}-\eqref{MP12} and \eqref{MP15}, we obtain \eqref{MP6}. Similarly, one can show \eqref{MP7}-\eqref{MP7''}.
\vskip2mm
\noindent
\textbf{Step 3:} In this step, we estimate each term of the right hand side of \eqref{MP5} separately.
\vskip2mm
\noindent
\textit{Estimate for $J_1(n,t)$:} \textit{When $d=2$ and $r=3$.} Using \eqref{441}, \eqref{b1}, Lemmas \ref{Holder} and \ref{Young}, we obtain
\begin{align}\label{MP16}
	\left|J_1(n,t)\right|\leq\int_{0}^{t}\left[\mu\|\u^n(\zeta)-\u(\zeta)\|^2_{\V}+C\|\u(\zeta)\|^2_{\H}\|\u(\zeta)\|^2_{\V}\|\u^n(\zeta)-\u(\zeta)\|^2_{\H}\right]\d\zeta.
\end{align}
\vskip2mm
\noindent
\textit{When $d=r=3$.} Using \eqref{441}, Lemmas \ref{Holder} and \ref{Young}, we obtain
\begin{align}\label{MP17}
	\left|J_1(n,t)\right|&\leq \int_{0}^{t}\left[2\theta\mu\|\u^n(\zeta)-\u(\zeta)\|^2_{\V}+\frac{1}{2\theta\mu}\|\u(\zeta)\big(\u^n(\zeta)-\u(\zeta)\big)\|^2_{\H}\right]\d\zeta,
\end{align}
for $\theta\in(0,1]$.\\
\vskip2mm
\noindent
\textit{When $d=2,3$ and $r>3$.} Using \eqref{441}, Lemmas \ref{Holder} and \ref{Young}, we obtain
\begin{align}\label{MP18}
\left|J_1(n,t)\right|\leq \int_{0}^{t}\left[\mu\|\u^n(\zeta)-\u(\zeta)\|^2_{\V}+\frac{\beta}{2}\||\u(\zeta)|^{\frac{r-1}{2}}\left|\u^n(\zeta)-\u(\zeta)\right|\|^2_{\H}+C\|\u^n(\zeta)-\u(\zeta)\|^2_{\H}\right]\d\zeta.
\end{align}
\vskip2mm
\noindent
\textit{Estimate for $J_2(n,t)$:} Using \eqref{MO_c} and \eqref{a215}, we get
\begin{align}\label{MP19}
	J_2(n,t)&=-2\beta\int_{0}^{t}\left\langle\mathcal{C}\big(\u^n(\zeta)\big)-\mathcal{C}\big(\u(\zeta)\big),\u^n(\zeta)-\u(\zeta)\right\rangle\d\zeta\nonumber\\&=-\beta\int_{0}^{t}\||\u^n(\zeta)|^{\frac{r-1}{2}}\left|\u^n(\zeta)-\u(\zeta)\right|\|^2_{\H}\d\zeta-\beta\int_{0}^{t}\||\u(\zeta)|^{\frac{r-1}{2}}\left|\u^n(\zeta)-\u(\zeta)\right|\|^2_{\H}\d\zeta\nonumber\\&\leq-\frac{\beta}{2}\int_{0}^{t}\||\u^n(\zeta)|^{\frac{r-1}{2}}\left|\u^n(\zeta)-\u(\zeta)\right|\|^2_{\H}\d\zeta-\frac{\beta}{2}\int_{0}^{t}\||\u(\zeta)|^{\frac{r-1}{2}}\left|\u^n(\zeta)-\u(\zeta)\right|\|^2_{\H}\d\zeta\nonumber\\&\quad-\frac{\beta}{2^{r-1}}\int_{0}^{t}\|\u^n(\zeta)-\u(\zeta)\|^{r+1}_{\wi\L^{r+1}}\d\zeta.
\end{align}
\vskip2mm
\noindent
\textit{Estimate for $J_3(n,t)$:} For $s\in[0,t\land\tau_{n}]$, we have
\begin{align}\label{MP20}
	&\left|J_3(n,s)\right|\nonumber\\&\leq2\int_{0}^{t\land\tau_{n}}\|\Big(\frac{1}{\sigma}\int\limits_{\lceil\frac{\xi}{\sigma}\rceil\sigma}^{(\lceil\frac{\xi}{\sigma}\rceil+1)\sigma}1_{\{\xi>t\land\tau_{n}\}}\d\xi\Big)\G\big(\u^n(\Big\lfloor\frac{\zeta}{\sigma}\Big\rfloor\sigma)\big)\|^2_{\mathcal{L}_2}\d\zeta\nonumber\\&\quad+4\int_{0}^{t\land\tau_{n}}\|\G\big(\u^n(\Big\lfloor\frac{\zeta}{\sigma}\Big\rfloor\sigma)\big)-\G\big(\u^n(\zeta)\big)\|^2_{\mathcal{L}_2}\d\zeta+8\int_{0}^{t\land\tau_{n}}\|\G\big(\u^n(\zeta)\big)\Pi_n-\G\big(\u^n(\zeta)\big)\|^2_{\mathcal{L}_2}\d\zeta\nonumber\\&\quad+8\int_{0}^{t\land\tau_{n}}\|\G\big(\u^n(\zeta)\big)-\G\big(\u(\zeta)\big)\|^2_{\mathcal{L}_2}\d\zeta\nonumber\\&\leq\underbrace{2\int_{t\land\tau_{n}-2\sigma}^{t\land\tau_{n}}\|\G\big(\u^n(\Big\lfloor\frac{\zeta}{\sigma}\Big\rfloor\sigma)\big)\|^2_{\mathcal{L}_2}\d\zeta}_{:=I_4(n)}+\underbrace{4\int_{0}^{\tau_{n}}\rho\big(\u^n(\zeta)\big)\|\u^n(\Big\lfloor\frac{\zeta}{\sigma}\Big\rfloor\sigma)-\u^n(\zeta)\|^2_{\H}\d\zeta}_{:=I_5(n)}\nonumber\\&\quad+\underbrace{8\int_{0}^{\tau_{n}}\|\G\big(\u^n(\zeta)\big)\Pi_n-\G\big(\u^n(\zeta)\big)\|^2_{\mathcal{L}_2}\d\zeta}_{:=I_6(n)}+8\int_{0}^{\tau_{n}}\rho\big(\u(\zeta)\big)\|\u^n(\zeta)-\u(\zeta)\|^2_{\H}\d\zeta.
\end{align}
From the definition of $\sigma$ (for $I_4(n)$), condition (H.2) of Hypothesis \ref{HonG}, \eqref{MP3} and \eqref{MP7} (for $I_5(n)$), and condition (H$'$.1) of Hypothesis \ref{HonG1} and \eqref{MP3} (for $I_6(n)$), we obtain
\begin{align}\label{MP21}
	\lim_{n\to\infty}I_7(n)=0, \ \ \ \text{ where } I_7(n)=\mathbb{E}\left[I_4(n)+I_5(n)+I_6(n)\right].
\end{align}
It implies from \eqref{MP20} that 
\begin{align}\label{MP22}
 \mathbb{E}\left[\sup_{t\in[0,\tau_{n}]}\left|J_3(n,t)\right|\right]\leq I_7(n)+8\mathbb{E}\left[\int_{0}^{\tau_{n}}\rho\big(\u(\zeta)\big)\|\u^n(\zeta)-\u(\zeta)\|^2_{\H}\d\zeta\right].
\end{align}
\vskip2mm
\noindent
\textit{Estimate for $J_4(n,t)$:} Making use of BDG inequality, we get
\begin{align}\label{MP23}
	\mathbb{E}\left[\sup_{t\in[0,\tau_{n}]}\left|J_4(n,t)\right|\right]&\leq4\mathbb{E}\left[\sup_{t\in[0,\tau_{n}]}\|\u^n(t)-\u(t)\|_{\H}\cdot\big\{\sup_{t\in[0,\tau_{n}]}\left|J_3(n,t)\right|\big\}^{\frac{1}{2}}\right]\nonumber\\&\leq\frac{1}{2}\mathbb{E}\left[\sup_{t\in[0,\tau_{n}]}\|\u^n(t)-\u(t)\|_{\H}^2\right]+8\mathbb{E}\left[\sup_{t\in[0,\tau_{n}]}\left|J_3(n,t)\right|\right].
\end{align} 
\vskip2mm
\noindent
\textit{Estimate for $J_5(n,t)$:} \textit{We claim that
\begin{align}\label{MP24}
	I_8(n):=\mathbb{E}\left[\sup_{t\in[0,\tau_{n}]}\left|J_5(n,t)\right|\right]\to 0 \text{ as } n\to \infty.
\end{align}}
The proof of \eqref{MP24} depends on obtaining an appropriate term from $$\big[\G\big(\u^n(\zeta)\big)-\G\big(\u^n\big((\Big\lfloor\frac{\zeta}{\sigma}\Big\rfloor-1)\sigma\big)\big)\big]\dot{\W}^n(\zeta),$$ which can be compensated with the correction term $-\frac{1}{2}\widetilde{\Tr}\big(\u^n(\zeta)\big)$. To find this appropriate term, by \eqref{AS1} and \eqref{AS2}, we equivalently write 
\begin{align}\label{MP25}
	\widetilde{\Tr}_n(\u^n)=\sum_{k=1}^{n}\D\G_k(\u^n)\G_k(\u^n), \ \ \G(\u^n)\dot{\W}^n=\sum_{k=1}^{n}\G_k(\u^n)\dot{w}_k^n.
\end{align}
Since $\G_k$ is twice Fr\'echet differentiable for all $k\in\N$, applying the second order Taylor's formula to $\G_k$, we obtain (Theorem 7.9.1, \cite{PGC})
\begin{align}\label{MP26}
	&\G_k\big(\u^n(\zeta)\big)-\G_k\big(\u^n\big((\Big\lfloor\frac{\zeta}{\sigma}\Big\rfloor-1)\sigma\big)\big)\nonumber\\&=\D\G_k\big(\u^n\big((\Big\lfloor\frac{\zeta}{\sigma}\Big\rfloor-1)\sigma\big)\big)\big[\u^n(\zeta)-\u^n\big((\Big\lfloor\frac{\zeta}{\sigma}\Big\rfloor-1)\sigma\big)\big]\nonumber\\&\quad+\int_{0}^{1}(1-\theta)\D^2\G_k\big(\theta\u^n(\zeta)+(1-\theta)\u^n\big((\Big\lfloor\frac{\zeta}{\sigma}\Big\rfloor-1)\sigma\big)\big)\d\theta\nonumber\\&\qquad\big\{\u^n(\zeta)-\u^n\big((\Big\lfloor\frac{\zeta}{\sigma}\Big\rfloor-1)\sigma\big), \u^n(\zeta)-\u^n\big((\Big\lfloor\frac{\zeta}{\sigma}\Big\rfloor-1)\sigma\big)\big\},
\end{align}
where $\D^2\G_k(\w)\{\w_1,\w_2\}$ represents the value of the second Fr\'echet derivative $\D^2\G_k(\w)$ on the elements $\w_1$ and $\w_2$. In view of \eqref{WZ_SCBF}, we can write 
\begin{align}\label{MP27}
	\u^n(\zeta)-\u^n\big((\Big\lfloor\frac{\zeta}{\sigma}\Big\rfloor-1)\sigma\big)&=-\int_{(\lfloor\frac{\zeta}{\sigma}\rfloor-1)\sigma}^{\zeta} \bigg\{\mu \A\u^n(\xi)+\B\big(\u^n(\xi)\big)+\alpha\u^n(\xi)+\beta\mathcal{C}\big(\u^n(\xi)\big)\nonumber\\&\qquad\qquad\qquad\qquad- \G\big(\u^n(\xi)\big)\dot{\W}^n(\xi)+\frac{1}{2}\widetilde{\Tr}_n\big(\u^n(\xi)\big)\bigg\}\d\xi,
\end{align}
in $\V'+\wi\L^{\frac{r+1}{r}}$. Making use of \eqref{AS1} and \eqref{MP25} (second equality), we have
\begin{align}\label{MP28}
	&\int_{(\lfloor\frac{\zeta}{\sigma}\rfloor-1)\sigma}^{\zeta}  \G\big(\u^n(\xi)\big)\dot{\W}^n(\xi)\d\xi\nonumber\\&=\sum_{j=1}^{n}\left[\dot{w}^n_j\big((\Big\lfloor\frac{\zeta}{\sigma}\Big\rfloor-1)\sigma\big)\int_{(\lfloor\frac{\zeta}{\sigma}\rfloor-1)\sigma}^{\lfloor\frac{\zeta}{\sigma}\rfloor\sigma}\G_j\big(\u^n(\xi)\big)\d\xi+\dot{w}^n_j(\zeta)\int_{\lfloor\frac{\zeta}{\sigma}\rfloor\sigma}^{\zeta}\G_j\big(\u^n(\xi)\big)\d\xi\right].
\end{align}
Now, putting the values from \eqref{MP25}-\eqref{MP28} into $J_5(n,t)$, we have
\begin{align}\label{MP29}
	J_5(n,t)=:S_1(n,t)+S_2(n,t)+S_3(n,t)+S_4(n,t)+S_5(n,t)+S_6(n,t),
\end{align}
where
\begin{align*}
	S_1(n,t):=&\sum_{k=1}^{n}\int_{0}^{t} \dot{w}^n_k(\zeta)\bigg\langle\D\G_k\big(\u^n\big((\Big\lfloor\frac{\zeta}{\sigma}\Big\rfloor-1)\sigma\big)\big)\int_{(\lfloor\frac{\zeta}{\sigma}\rfloor-1)\sigma}^{\zeta}\big\{\mu \A\u^n(\xi)+\B\big(\u^n(\xi)\big)+\alpha\u^n(\xi)\nonumber\\&\qquad\qquad\qquad+\beta\mathcal{C}\big(\u^n(\xi)\big)\big\}\d\xi,\u^n(\zeta)-\u(\zeta)\bigg\rangle\d\zeta,\\
	S_2(n,t):=&\sum_{k=1}^{n}\sum_{j=1}^{n}\int_{0}^{t}\dot{w}^n_k(\zeta)\dot{w}^n_j\big((\Big\lfloor\frac{\zeta}{\sigma}\Big\rfloor-1)\sigma\big)\bigg(\D\G_k\big(\u^n\big((\Big\lfloor\frac{\zeta}{\sigma}\Big\rfloor-1)\sigma\big)\big)\\&\qquad\qquad\qquad\times\int_{(\lfloor\frac{\zeta}{\sigma}\rfloor-1)\sigma}^{\lfloor\frac{\zeta}{\sigma}\rfloor\sigma}\G_j\big(\u^n(\xi)\big)\d\xi,\u^n(\zeta)-\u(\zeta)\bigg)\d\zeta,\\
	S_3(n,t):=&\sum_{k=1}^{n}\sum_{1\leq j\leq n, j\neq k}^{n}\int_{0}^{t}\dot{w}^n_k(\zeta)\dot{w}^n_j(\zeta)\bigg(\D\G_k\big(\u^n\big((\Big\lfloor\frac{\zeta}{\sigma}\Big\rfloor-1)\sigma\big)\big)\\&\qquad\qquad\qquad\times\int_{\lfloor\frac{\zeta}{\sigma}\rfloor\sigma}^{\zeta}\G_j\big(\u^n(\xi)\big)\d\xi,\u^n(\zeta)-\u(\zeta)\bigg)\d\zeta,\\
	S_4(n,t):=&\sum_{k=1}^{n}\int_{0}^{t}\bigg([\dot{w}^n_k(\zeta)]^2\D\G_k\big(\u^n\big((\Big\lfloor\frac{\zeta}{\sigma}\Big\rfloor-1)\sigma\big)\big)\int_{\lfloor\frac{\zeta}{\sigma}\rfloor\sigma}^{\zeta}\G_k\big(\u^n(\xi)\big)\d\xi\\&\qquad\qquad\qquad-\frac{1}{2}\D\G_k\big(\u^n(s)\big)\G_k\big(\u^n(s)\big),\u^n(\zeta)-\u(\zeta)\bigg)\d\zeta,\\
	S_5(n,t):=&-\frac{1}{2}\sum_{k=1}^{n}\int_{0}^{t}\dot{w}^n_k(\zeta)\bigg(\D\G_k\big(\u^n\big((\Big\lfloor\frac{\zeta}{\sigma}\Big\rfloor-1)\sigma\big)\big)\\&\qquad\qquad\qquad\times\int_{(\lfloor\frac{\zeta}{\sigma}\rfloor-1)\sigma}^{\zeta}\widetilde{\Tr}_n\big(\u^n(\xi)\big)\d\xi,\u^n(\zeta)-\u(\zeta)\bigg)\d\zeta,\\
	S_6(n,t):=&\sum_{k=1}^{n}\int_{0}^{t}\dot{w}^n_k(\zeta)\bigg(\int_{0}^{1}(1-\theta)\D^2\G_k\big(\theta\u^n(\zeta)+(1-\theta)\u^n\big((\Big\lfloor\frac{\zeta}{\sigma}\Big\rfloor-1)\sigma\big)\big)\d\theta\nonumber\\&\qquad\big\{\u^n(\zeta)-\u^n\big((\Big\lfloor\frac{\zeta}{\sigma}\Big\rfloor-1)\sigma\big), \u^n(\zeta)-\u^n\big((\Big\lfloor\frac{\zeta}{\sigma}\Big\rfloor-1)\sigma\big)\big\},\u^n(\zeta)-\u(\zeta)\bigg)\d\zeta.
\end{align*}
Next, we estimate $S_{i}(n,t)$ for $i=1,2,\ldots,6,$ separately. For $S_1(n,t)$, we have 
\begin{align*}
	S_1(n,t)=&\sum_{k=1}^{n}\int_{0}^{t} \dot{w}^n_k(\zeta)\int_{(\lfloor\frac{\zeta}{\sigma}\rfloor-1)\sigma}^{\zeta}\bigg\langle\big\{\mu\A\u^n(\xi)+\B\big(\u^n(\xi)\big)+\alpha\u^n(\xi)\nonumber\\&\qquad\qquad\qquad+\beta\mathcal{C}\big(\u^n(\xi)\big)\big\},\D\G_k\big(\u^n\big((\Big\lfloor\frac{\zeta}{\sigma}\Big\rfloor-1)\sigma\big)\big)^{*}[\u^n(\zeta)-\u(\zeta)]\bigg\rangle\d\xi\d\zeta.
\end{align*}
Using \eqref{MP3},  condition (H$'$.1) of Hypothesis \ref{HonG1}, Lemmas \ref{Holder} and \ref{Young}, and Remark \ref{RemarkB}, we obtain
\begin{align*}
	&\mathbb{E}\left[\sup_{t\in[0,\tau_{n}]}\left|S_1(n,t)\right|\right]\nonumber\\&\leq C(M)\sum_{k=1}^{n}\mathbb{E}\bigg[\int_{0}^{\tau_n}\left|\dot{w}^n_k(\zeta)\right|\int_{(\lfloor\frac{\zeta}{\sigma}\rfloor-1)\sigma}^{\zeta}\bigg\{\big(\|\u^n(\xi)\|_{\V}+\|\B(\u^n(\xi))\|_{\V^{'}}\big)\|\u^n(\zeta)-\u(\zeta)\|_{\V}\nonumber\\&\qquad\qquad+\|\u^n(\xi)\|^{r}_{\wi\L^{r+1}}\|\u^n(\zeta)-\u(\zeta)\|_{\wi\L^{r+1}}\bigg\}\d\xi\d\zeta\bigg]\nonumber\\&\leq C(M)n^{\frac{3}{2}}2^{\frac{n}{2}}\cdot\mathbb{E}\bigg[\int_{0}^{\tau_{n}}\bigg\{\int_{(\lfloor\frac{\zeta}{\sigma}\rfloor-1)\sigma}^{\zeta}\bigg(\|\u^n(\xi)\|^{2}_{\V}+\|\u^n(\xi)\|^{r+1}_{\wi\L^{r+1}}\bigg)\d\xi\bigg\}^{\frac{1}{2}}\nonumber\\&\qquad\qquad\qquad\times\bigg\{\int_{(\lfloor\frac{\zeta}{\sigma}\rfloor-1)\sigma}^{\zeta}\|\u^n(\zeta)-\u(\zeta)\|^{2}_{\V}\d\xi\bigg\}^{\frac{1}{2}}\d\zeta+\int_{0}^{\tau_{n}}\bigg\{\int_{(\lfloor\frac{\zeta}{\sigma}\rfloor-1)\sigma}^{\zeta}\|\u^n(\xi)\|^{r+1}_{\wi\L^{r+1}}\d\xi\bigg\}^{\frac{r}{r+1}}\nonumber\\&\qquad\qquad\qquad\times\bigg\{\int_{(\lfloor\frac{\zeta}{\sigma}\rfloor-1)\sigma}^{\zeta}\|\u^n(\zeta)-\u(\zeta)\|^{r+1}_{\wi\L^{r+1}}\d\xi\bigg\}^{\frac{1}{r+1}}\d\zeta\bigg]\nonumber\\&\leq C(M)n^{\frac{3}{2}}2^{\frac{n}{2}}\sigma^{\frac{1}{2}}\bigg(\mathbb{E}\bigg[\int_{0}^{\tau_{n}}\int_{(\lfloor\frac{\zeta}{\sigma}\rfloor-1)\sigma}^{\zeta}\bigg\{\|\u^n(\xi)\|^{2}_{\V}+\|\u^n(\xi)\|^{r+1}_{\wi\L^{r+1}}\bigg\}\d\xi\d\zeta\bigg]\bigg)^{\frac{1}{2}}\nonumber\\&\quad+C(M)n^{\frac{3}{2}}2^{\frac{n}{2}}\sigma^{\frac{1}{r+1}}\bigg(\mathbb{E}\bigg[\int_{0}^{\tau_{n}}\int_{(\lfloor\frac{\zeta}{\sigma}\rfloor-1)\sigma}^{\zeta}\|\u^n(\xi)\|^{r+1}_{\wi\L^{r+1}}\d\xi\d\zeta\bigg]\bigg)^{\frac{r}{r+1}}.
\end{align*}
Applying stochastic Fubini's Theorem and \eqref{E_U2}, we obtain 
\begin{align}\label{MP30}
	&\mathbb{E}\left[\sup_{t\in[0,\tau_{n}]}\left|S_1(n,t)\right|\right]\leq C(M)n^{\frac{3}{2}}2^{\frac{n}{2}}\sigma= C(T,M)n^{\frac{3}{2}}2^{-\frac{n}{2}}.
\end{align}
For $S_{i}(n,t)$ for $i=2,3,\ldots,6,$ we refer the readers to the work \cite{TMRZ} (see proof of (2.19) in section 2 of \cite{TMRZ}). Since, the calculations are same, therefore we are not repeating here. But for the completeness, we provide bounds for each $S_{i}(n,t)$ for $i=2,3,\ldots,6$.
 \begin{equation}\label{MP31}
	\left\{
	\begin{aligned}
	\mathbb{E}\left[\sup_{t\in[0,\tau_{n}]}\left|S_2(n,t)\right|\right]\leq C(T,M)n^{3}2^{-\frac{3n}{8}},\\ 
		\mathbb{E}\left[\sup_{t\in[0,\tau_{n}]}\left|S_3(n,t)\right|\right]\leq C(T,M)n^{3}2^{-\frac{3n}{8}},\\
			\mathbb{E}\left[\sup_{t\in[0,\tau_{n}]}\left|S_4(n,t)\right|\right]\leq C(T,M)n^{2}2^{-\frac{3n}{8}},\\ 
				\mathbb{E}\left[\sup_{t\in[0,\tau_{n}]}\left|S_5(n,t)\right|\right]\leq C(T,M)n^{\frac{3}{2}}2^{-\frac{n}{2}},\\
					\mathbb{E}\left[\sup_{t\in[0,\tau_{n}]}\left|S_6(n,t)\right|\right]\leq C(T,M)n^{\frac{3}{2}}2^{-\frac{n}{4}}.
 \end{aligned}
\right.
\end{equation}
Combining \eqref{MP29}-\eqref{MP31} and passing limit $n\to\infty$, we obtain the convergence claimed in \eqref{MP24}.
\vskip2mm
\noindent
\textbf{Step 4:} Taking supremum and expectation, respectively, both sides of \eqref{MP5}, and using \eqref{MP16}-\eqref{MP19} and \eqref{MP22}-\eqref{MP24} in final estimate, we obtain
\begin{align}\label{MP34}
	&\mathbb{E}\left[\frac{1}{2}\sup_{\zeta\in[0,\tau_{n}]}\|\u^n(\zeta)-\u(\zeta)\|^2_{\H}+a_1\int_{0}^{\tau_{n}}\|\u^n(\zeta)-\u(\zeta)\|^2_{\V}\d\zeta+a_2\int_{0}^{\tau_{n}}\|\u^n(\zeta)-\u(\zeta)\|^{r+1}_{\wi\L^{r+1}}\d\zeta\right]\nonumber\\&\leq I_7(n)+I_8(n)+\mathbb{E}\left[C\int_{0}^{\tau_{n}}\big(\Phi(\u(\zeta))+\rho(\u(\zeta))\big)\|\u^n(\zeta)-\u(\zeta)\|^2_{\H}\d\zeta\right],
\end{align}
where, for $\theta\in(0,1]$,
\begin{align}\label{MP35}
	a_1=\begin{cases}
		\mu  &  \text{ for } d=2 \text{ and } r=3,\\2\mu(1-\theta)& \text{ for } d=r=3 ,\\ \mu& \text{ for } d=2,3 \text{ and } r>3,
	\end{cases}\ 
a_2=\begin{cases}
	\frac{\beta}{2^{r-2}} & \text{ for } d=2 \text{ and } r=3,\\ \frac{1}{2^{r-2}}\left(\beta-\frac{1}{2\theta\mu}\right) & \text{ for } d=r=3 ,\\ \frac{\beta}{2^{r-1}}& \text{ for } d=2,3 \text{ and } r>3,
\end{cases}
\end{align}
and
\begin{align}\label{MP36}
	\Phi(\u(\zeta))=C\begin{cases}
		 \|\u(\zeta)\|^2_{\H}\|\u(\zeta)\|^2_{\V} &\ \ \text{ for } d=2 \text{ and } r=3,\\ 0 &\ \ \text{ for } d=r=3 ,\\ 1&\ \ \text{ for } d=2,3 \text{ and } r>3.
	\end{cases}
\end{align}
Now, set $\mathfrak{X}(t):=\sup_{\zeta\in[0,t]}\|\u^n(\zeta)-\u(\zeta)\|^2_{\H}$ and $\mathfrak{Z}(t):=C\int_{0}^{t}\big(\Phi(\u(\zeta))+\rho(\u(\zeta))\big)\d\zeta,$ for $t\in[0,T]$. Then $\mathfrak{X}$ and $\mathfrak{Z}$ are adapted, non-negative and continuous. From \eqref{Tau1}, condition (H.1) of Hypothesis \ref{HonG} and \eqref{MP36}, we obtain that there exists a constant $C'(M)$ such that $\mathfrak{Z}(t)\leq C'(M)$ uniformly for $t\in[0,\tau_{n}]$. From \eqref{MP34}, we have
\begin{align}\label{MP37}
	\mathbb{E}\bigg[\mathfrak{X}(\tau_{n})\bigg]\leq I_7(n)+I_8(n)+\mathbb{E}\left[\int_{0}^{\tau_{n}}\mathfrak{X}(s)\d\mathfrak{Z}(s)\right].
\end{align}
Applying Lemma 2 from \cite{GM} to \eqref{MP37}, we achieve
\begin{align}\label{MP38}
	\mathbb{E}\left[\int_{0}^{\tau_{n}}\mathfrak{X}(s)\d\mathfrak{Z}(s)\right]\leq \left\{I_7(n)+I_8(n)\right\}e^{C'(M)}\int_{0}^{C'(M)}e^{-y}\d y\to 0 \text{ as } n\to \infty,
\end{align}
where we have used the convergences obtained in \eqref{MP21} and \eqref{MP24}. Hence, \eqref{MP37} together with \eqref{MP21}, \eqref{MP24} and \eqref{MP38} provide  that
\begin{align*}
	\lim_{n\to\infty}\mathbb{E}\left[\sup_{\zeta\in[0,\tau_{n}]}\|\u^n(\zeta)-\u(\zeta)\|^2_{\H}\right]=0,
\end{align*}
as required.
\end{proof}
\section{Support of solutions of stochastic CBF equations}\label{sec4}\setcounter{equation}{0}
This section is devoted to establish an application of Wong-Zakai approximation, that is, the support of solutions of stochastic CBF equation \eqref{SCBF}. Assume that $T>0$ and $\W$ be a cylindrical process discussed in subsection \ref{AF}. For $n\in\N$, $t\in[0,T],$ define
\begin{align*}
	M^n(t):= \exp\left(\int_{0}^{t}\dot{\W}^n(s)\d\W(s)-\frac{1}{2}\int_{0}^{t}\|\dot{\W}^n(s)\|^2_{\mathrm{K}}\d s\right),
\end{align*}
where $\dot{\W}^n$ is defined in \eqref{AS1} and 
\begin{align}\label{AS3}
	\widehat{\W}^n(t):= \W(t)-\int_{0}^{t}\dot{\W}^n(s)\d s.
\end{align}
Since the real valued random variables $\dot{w}_k(j\sigma)$, $j,k\in\N$ are independent and for each $j,k\in\N$, $\sigma^{\frac{1}{2}}\dot{w}_k(j\sigma)$ is standard normal. Therefore for all $n\in\N$
\begin{align*}
	\sup_{t\in[0,T]}\mathbb{E}\left[e^{\lambda\|\dot{\W}^n(t)\|^2_{\mathrm{K}}}\right]=\sup_{t\in[0,T]}\prod_{1\leq k\leq n}\mathbb{E}\left[e^{\lambda\left|\dot{w}_k^n(t)\right|^2}\right]=\left(\mathbb{E}\left[e^{\frac{\lambda}{\sigma}\left|Z\right|^2}\right]\right)^n<\infty
\end{align*} 
is satisfied for some standard normal random variable $Z$ and $\lambda>0$ sufficiently small (Fernique's theorem). Consequently, by Girsanov's theorem (Theorem 10.14 and Proposition 10.17, \cite{DZ1}), the process $\{\widehat{\W}^n(t)\}_{t\in[0,T]}$ given by \eqref{AS3} is a cylindrical Wiener process under $\mathbb{P}^n$ with the measure $\mathbb{P}^n\ll\mathbb{P}$ satisfying 
\begin{align*}
	\frac{\d\mathbb{P}^n}{\d\mathbb{P}}\bigg|_{\mathscr{F}_t}=M^n(t), \text{ for } t\in[0,T].
\end{align*}
Analogously, for arbitrary $\k\in\mathrm{L}^2(0,T;\mathrm{K})$, we define the process
\begin{align*}
	M^n_{\k}:=\exp\left(-\int_{0}^{t}\k(s)\d\widehat{\W}^n(s)-\frac{1}{2}\int_{0}^{t}\|\k(s)\|^2_{\mathrm{K}}\d s\right),
\end{align*}
and
\begin{align}\label{AS4}
	\widehat{\W}^n_{\k}(t):=\widehat{\W}^n(t)+\int_{0}^{t}\k(s)\d s, \text{ for } t\in[0,T], n\in\N.
\end{align}
Again, in view of Girsanov's theorem, we get another probability measure $\mathbb{P}^n_{\k}\ll\mathbb{P}^n\ll\mathbb{P}$ such that 
\begin{align}\label{AS5}
	\frac{\d\mathbb{P}^n_{\k}}{\d\mathbb{P}^n}\bigg|_{\mathscr{F}_t}=M^n_{\k}(t), \text{ for } t\in[0,T],
\end{align}
where $\widehat{\W}^n_{\k}$ is a cylindrical Wiener process under $\mathbb{P}^n_{\k}$. 

Now, for fixed $T>0$, consider the system
	\begin{equation}\label{C_SCBF}
	\left\{
	\begin{aligned}
	\d\mathrm{U}^n_{\k}(t)&=-\left[\mu\A\mathrm{U}^n_{\k}(t)+\B(\mathrm{U}^n_{\k}(t))+\alpha\mathrm{U}^n_{\k}(t)+\mathcal{C}(\mathrm{U}^n_{\k}(t))\right]\d t + \F_1(\mathrm{U}^n_{\k}(t))\d\W(t) \\&\quad+\F_2(\mathrm{U}^n_{\k}(t))\dot{\W}^n(t)\d t+\F_3(\mathrm{U}^n_{\k}(t))\k(t)\d t- \F(\mathrm{U}^n_{\k}(t))\d t,\\
		\mathrm{U}^n_{\k}(0)&=\x,
	\end{aligned}
	\right.
\end{equation}
where $\F_1,\F_2,\F_3:\H\times\Omega\to(\mathcal{L}_2(\mathrm{K};\H),\|\cdot\|_{\mathcal{L}_2})$ and $\F:\H\times\Omega\to\H$ are progressively measurable.
 Note that, the system \eqref{WZ_SCBF} is a spacial case of system \eqref{C_SCBF} with $\F_1=0$, $\F_2=\G$, $\F_3=0$ and $\F=\frac{1}{2}\widetilde{\Tr}_n$. Hence, the solvability of system \eqref{WZ_SCBF}  is an immediate consequence of the solvability of the system \eqref{C_SCBF}. Let us state the result for existence of a unique solution for  the system \eqref{C_SCBF} which is proved in Section \ref{sec5} (see below).
 
 \begin{theorem}\label{4.1}
 	Let $T>0$, $\k\in\mathrm{L}^2(0,T;\mathrm{K})$, $p>\gamma+2$ and $\x\in\mathrm{L}^{p}(\Omega,\mathscr{F},\mathbb{P};\H)$, where $\gamma$ as in Hypothesis \ref{HonG}. Assume that the operators $\F_1,\F_2$ and $\F_3$ satisfy the conditions of Hypothesis \ref{HonG} and $\F$ satisfies the condition \emph{(H$'$.2)} of Hypothesis \ref{HonG1}. Then there exists a unique solution $\mathrm{U}^n_{\k}$  to equation \eqref{C_SCBF} with the initial condition $\x$. Furthermore, $\mathrm{U}^n_{\k}\in\mathrm{L}^{p}(\Omega; \mathrm{L}^{\infty}(0,T;\H))\cap\mathrm{L}^2(\Omega; \mathrm{L}^2(0,T;\V))\cap\mathrm{L}^{r+1}(\Omega;\mathrm{L}^{r+1}(0,T;\widetilde{\L}^{r+1}))$, that is,
 	\begin{align}\label{E_U6}
 		\sup_{n\geq1}\mathbb{E}\left[\sup_{t\in[0,T]}\|\mathrm{U}^n_{\k}(t)\|^{p}_{\H}+\int_{0}^{T}\|\mathrm{U}^n_{\k}(t)\|^2_{\V}\d t+\int_{0}^{T}\|\mathrm{U}^n_{\k}(t)\|^{r+1}_{\wi\L^{r+1}}\d t\right]<\infty,
 	\end{align} 
 with a modification having paths in $\C([0,T];\H)\cap\mathrm{L}^2(0,T;\V)\cap\mathrm{L}^{r+1}(0,T;\wi\L^{r+1}),$ $\mathbb{P}$-a.s.
 \end{theorem}

 For $\k\in\mathrm{L}^2(0,T;\mathrm{K})$, let us consider two systems which can be seen as two spacial cases of the system \eqref{C_SCBF},
	\begin{equation}\label{C_SCBF1}
	\left\{
	\begin{aligned}
		\d\mathrm{Y}_{\k}(t)&=-\left\{\mu\A\mathrm{Y}_{\k}(t)+\B(\mathrm{Y}_{\k}(t))+\alpha\mathrm{Y}_{\k}(t)+\mathcal{C}(\mathrm{Y}_{\k}(t))\right\}\d t +\G(\mathrm{Y}_{\k}(t))\k(t)\d t\\&\quad- \frac{1}{2}\widetilde{\Tr}(\mathrm{Y}_{\k}(t))\d t,\\
		\mathrm{Y}_{\k}(0)&=\x,
	\end{aligned}
	\right.
\end{equation}
 and
 	\begin{equation}\label{C_SCBF2}
 	\left\{
 	\begin{aligned}
 		\d\mathrm{Y}^n_{\k}(t)&=-\left\{\mu\A\mathrm{Y}^n_{\k}(t)+\B(\mathrm{Y}^n_{\k}(t))+\alpha\mathrm{Y}^n_{\k}(t)+\mathcal{C}(\mathrm{Y}^n_{\k}(t))\right\}\d t + \G(\mathrm{Y}^n_{\k}(t))\d\W(t) \\&\quad-\G(\mathrm{Y}^n_{\k}(t))\dot{\W}^n(t)\d t+\G(\mathrm{Y}^n_{\k}(t))\k(t)\d t,\\
 		\mathrm{Y}^n_{\k}(0)&=\x,
 	\end{aligned}
 	\right.
 \end{equation}
  where the operators $\A$, $\B$, $\mathcal{C}$, $\G$ and $\widetilde{\Tr}$ are defined in section \ref{sec2}. Furthermore, the existence of unique solutions $\mathrm{Y}_{\k}$ and $\mathrm{Y}^n_{\k}$ of the systems \eqref{C_SCBF1} and \eqref{C_SCBF2}, respectively, are confirmed from Theorem \ref{4.1}. Moreover, we have $\mathrm{Y}_{\k}$, $\mathrm{Y}^n_{\k}\in\mathrm{C}([0,T];\H))\cap\mathrm{L}^2(0,T;\V)\cap\mathrm{L}^{r+1}(0,T;\widetilde{\L}^{r+1})$, $\mathbb{P}$-a.s. and
  \begin{align}\label{E_U7}
  	\mathbb{E}\left[\sup_{t\in[0,T]}\|\mathrm{Y}_{\k}(t)\|^{p}_{\H}+\int_{0}^{T}\|\mathrm{Y}_{\k}(t)\|^2_{\V}\d t+\int_{0}^{T}\|\mathrm{Y}_{\k}(t)\|^{r+1}_{\wi\L^{r+1}}\d t\right]<\infty,
  \end{align} 
and
\begin{align}\label{E_U8}
	\sup_{n\geq1}\mathbb{E}\left[\sup_{t\in[0,T]}\|\mathrm{Y}^n_{\k}(t)\|^{p}_{\H}+\int_{0}^{T}\|\mathrm{Y}^n_{\k}(t)\|^2_{\V}\d t+\int_{0}^{T}\|\mathrm{Y}^n_{\k}(t)\|^{r+1}_{\wi\L^{r+1}}\d t\right]<\infty.
\end{align} 
The following Lemma demonstrates the Wong-Zakai approximation results for the systems \eqref{C_SCBF1} and \eqref{C_SCBF2}.
\begin{lemma}\label{STWZ}
		Assume that all the conditions of Hypotheses \ref{HonG} and \ref{HonG1} are satisfied, $p>\gamma+2$ and $\x\in \mathrm{L}^{p}(\Omega, \mathscr{F}_0,\mathbb{P};\H)$. For $\k\in\mathrm{L}^2(0,T;\mathrm{K})$, let $\mathrm{Y}_{\k}$ and $\mathrm{Y}^n_{\k}$ be the solutions to the systems \eqref{C_SCBF1} and \eqref{C_SCBF2}, respectively, with same initial data $\x$. Then
	\begin{align}\label{STWZ1}
		\lim_{n\to\infty}\mathbb{E}\bigg[\sup_{t\in[0,T]}\|\mathrm{Y}^n_{\k}(t)-\mathrm{Y}_{\k}(t)\|^2_{\H}\bigg]=0.
	\end{align}
\end{lemma}
\begin{proof}
	The proof of this lemma is similar to the proof of Theorem \ref{Main}. In comparison with \eqref{MP5}, the only term which we need to control is 
	\begin{align*}
&\int_{0}^{t}\big(\big[\G(\Y^n_{\k}(\zeta))-\G(\Y_{\k}(\zeta))\big]\k(\zeta),\Y^n_{\k}(\zeta)-\Y_{\k}(\zeta)\big)\d\zeta\\&\leq\int_{0}^{t}\left(\rho(\Y_{\k}(\zeta))+\|\k(\zeta)\|^2_{\mathrm{K}}\right)\|\Y^n_{\k}(\zeta)-\Y_{\k}(\zeta)\|^2_{\H}\d\zeta.
	\end{align*}
Since $\k\in\mathrm{L}^2(0,T;\mathrm{K})$, applying similar arguments as in the proof of Theorem \ref{Main}, we obtain \eqref{STWZ1}, as desired.
\end{proof}
Let $\mathcal{D}= \mathrm{C}([0,T];\H))$. Define $\mathcal{L}:=\{\Y_{\k}, \ \k\in\mathrm{L}^2(0,T;\mathrm{K})\}$ and observe that $\mathcal{L}\subset\mathcal{D}$. Let us now state and prove the support theorem. 
\begin{theorem}\label{ST}
	Assume that $\x\in \mathrm{L}^{p}(\Omega, \mathscr{F}_0,\mathbb{P};\H)$ with $p>\gamma+2$, where $\gamma$ is the same as in Hypothesis \ref{HonG}, and Hypotheses \ref{HonG} and \ref{HonG1} are satisfied. Let $\u$ represent the solution of the system \eqref{SCBF} with initial data $\x$. Then $$\mathrm{supp}(\mathbb{P}\circ\u^{-1})=\bar{\mathcal{L}},$$ where $\bar{\mathcal{L}}$ is the closure of $\mathcal{L}$ in $\mathcal{D}$ and $\mathrm{supp}(\mathbb{P}\circ\u^{-1})$ represents the support of the distribution $\mathbb{P}\circ\u^{-1}$.
\end{theorem}
\begin{proof}
	Let $\u$ and $\u^n$ be the solutions of the systems \eqref{SCBF} and \eqref{WZ_SCBF}, respectively. Since $\Y_{\k}$ represents the solution of the system \eqref{C_SCBF1}, then, for $\k=\dot{\W}^n$, we have $\Y_{\dot{\W}^n}=\u^n$, $\mathbb{P}$-a.s. Furthermore, it implies from Theorem \ref{Main} that for every $\nu>0$,
\begin{align*}
	\lim_{n\to\infty}\mathbb{P}\left(\|\Y_{\dot{\W}^n}-\u\|_{\mathcal{D}}\geq\nu\right)=	\lim_{n\to\infty}\mathbb{P}\left(\|\u^n-\u\|_{\mathcal{D}}\geq\nu\right)=0.
\end{align*} 
Since $\mathbb{P}$-a.s. $\dot{\W}^n\in\mathrm{L}^2(0,T;\mathrm{K})$, we have
\begin{align}\label{ST1}
	\text{supp}(\mathbb{P}\circ\u^{-1})\subset\bar{\mathcal{L}}.
\end{align}

Conversely, we claim that $\text{supp}(\mathbb{P}\circ\u^{-1})\supset\bar{\mathcal{L}}.$ By Remark 2.5.1 of \cite{LR}, there exists an another Hilbert space $\bar{\mathrm{K}}\supset\mathrm{K}$ such that the embedding from $(\mathrm{K},\langle\cdot,\cdot\rangle_{\mathrm{K}})$ to  $(\bar{\mathrm{K}},\langle\cdot,\cdot\rangle_{\bar{\mathrm{K}}})$ is Hilbert-Schmidt. Additionally, we can find a set $\{e_j, j\in\N\}\subseteq\mathrm{K},$ and $0<\lambda_1\leq\lambda_2\leq\cdots\leq\lambda_j\uparrow\infty$ such that $\{e_j, j\in\N\}$ is an orthonormal basis in $\mathrm{K}$ and $\{\sqrt{\lambda_j}e_j, j\in\N\}$ is an orthonormal basis in $\bar{\mathrm{K}}$. Fix such $\bar{\mathrm{K}}$ and define $\mathbb{W}^{\bar{\mathrm{K}}}:=\mathrm{C}([0,\infty);\bar{\mathrm{K}})$ and $\mathbb{W}_0^{\bar{\mathrm{K}}}:=\{w\in\mathbb{W}^{\bar{\mathrm{K}}}|w(0)=0\}$. Here  $\mathbb{W}_0^{\bar{\mathrm{K}}}$ is equipped with metric
\begin{align*}
	d(w^1,w^2):=\sum_{j=1}^{\infty}\left(\max_{0\leq t\leq j}\|w^1(t)-w^2(t)\|_{\bar{\mathrm{K}}}\land1\right), \ \ w^1,w^2\in\mathbb{W}_0^{\bar{\mathrm{K}}}.
\end{align*} 
Then $\mathbb{W}_0^{\bar{\mathrm{K}}}$ is a complete metrizable space with respect to the metric $d(\cdot,\cdot)$. Let  $\mathscr{B}(\mathbb{W}_0^{\bar{\mathrm{K}}})$ represents its Borel sigma-algebra. It follows that $\W\in\mathbb{W}_0^{\bar{\mathrm{K}}}$ $\mathbb{P}$-a.s. Let $\{\mathscr{B}_t(\mathbb{W}_0^{\bar{\mathrm{K}}})\}_{t\geq0}$ be the normal filtration generated by the canonical process $\omega$. It gives another complete probability space 
\begin{align*}
	\left(\mathbb{W}_0^{\bar{\mathrm{K}}},\cup_{t\geq0}\mathscr{B}_t(\mathbb{W}_0^{\bar{\mathrm{K}}}),\mathscr{B}_t(\mathbb{W}_0^{\bar{\mathrm{K}}}), \bar{\mathbb{P}}\right),
\end{align*}
where $\bar{\mathbb{P}}$ represents the distribution of $\omega$ in $\mathbb{W}^{\bar{\mathrm{K}}}$, that is,
\begin{align}\label{ST2}
	\bar{\mathbb{P}}\circ\omega^{-1}=\mathbb{P}\circ\W^{-1}.
\end{align}
Let $\x$ be $\mathscr{F}_0/\mathscr{B}(\H)$-measurable and $\x\in \mathrm{L}^{p}(\Omega, \mathscr{F}_0,\mathbb{P};\H)$ with $p>\gamma+2$. By Proposition \ref{E_U1} and the Yamada-Watanabe Theorem in \cite{PR} (Theorem E.1.8, \cite{PR}), we can find a measurable map
\begin{align*}
	\mathcal{S}_{\mathbb{P}\circ\x^{-1}}:(\H\times\mathbb{W}_0^{\bar{\mathrm{K}}},\mathscr{B}(\H)\otimes\mathscr{B}(\mathbb{W}_0^{\bar{\mathrm{K}}}))\to (\X,\mathscr{B}(\X)),
\end{align*}
where $\X=\mathrm{C}([0,T];\H))$, such that $\u:=\mathcal{S}_{\mathbb{P}\circ\x^{-1}}(\x,\W)$ is the solution of the system \eqref{SCBF} with the initial data $\u(0)=\x,$ $\mathbb{P}$-a.s. For $\k\in\mathrm{L}^2(0,T;\mathrm{K})$, define maps $T^n_{\k}$ on $(\mathbb{W}_0^{\bar{\mathrm{K}}},\mathscr{B}(\mathbb{W}_0^{\bar{\mathrm{K}}}))$ by
\begin{align*}
	T^n_{\k}(w)=w-\int_0^{\cdot}\dot{w}^n(\zeta)\d\zeta+\int_{0}^{\cdot}\k(\zeta)\d\zeta, \ \ w\in\mathbb{W}_0^{\bar{\mathrm{K}}},
\end{align*}
where
\begin{align*}
	\dot{w}^n(t):=\sum_{j=1}^{n}\Big\langle\frac{\lambda_j}{\sigma}\left[w(\Big\lfloor\frac{t}{\sigma}\Big\rfloor\sigma)-w((\Big\lfloor\frac{t}{\sigma}\Big\rfloor-1)\sigma)\right],e_j\Big\rangle_{\bar{\mathrm{K}}}e_j ,\ \ t\in[0,T].
\end{align*}
We infer from \eqref{AS3}-\eqref{AS5} that $T^n_{\k}$ can be seen as measurable transformations of Wiener space $\mathbb{W}_0^{\bar{\mathrm{K}}}$. Select a $\mathscr{B}(\mathbb{W}_0^{\bar{\mathrm{K}}})/\mathscr{B}(\H)$-measurable map $\x_0:\mathbb{W}_0^{\bar{\mathrm{K}}}\to\H$ such that $\bar{\mathbb{P}}\circ\x_0^{-1}=\mathbb{P}\circ\x^{-1}$. Then $\mathcal{S}_{\bar{\mathbb{P}}\circ\x_0^{-1}}(\x_0(\omega),\omega)$ is also a solution of the system \eqref{SCBF} with the initial data $\x_0$ and noise $\omega$. Since $\x_0$ is $\mathscr{B}(\mathbb{W}_0^{\bar{\mathrm{K}}})/\mathscr{B}(\H)$-measurable, $\x_0(T^n_{\k}(\omega))=\x_0(\omega)$, for all $n$. Again, from Yamada-Watanabe theorem, pathwise uniqueness implies that for every $\nu>0$, $n\in\N$
\begin{align}\label{ST3}
	\bar{\mathbb{P}}\left(\omega:\|\mathcal{S}_{\bar{\mathbb{P}}\circ\x_0^{-1}}(\x_0(\omega),T^n_{\k}(\omega))-\Y_{\k}\|_{\mathcal{D}}\geq\nu\right)=\mathbb{P}(\|\Y^n_{\k}-\Y_{\k}\|_{\mathcal{D}}\geq\nu),
\end{align}
which implies from Lemma \ref{STWZ} that for $\nu$ in \eqref{ST3}
\begin{align}\label{ST4}
	\lim_{n\to\infty}\mathbb{P}(\|\Y^n_{\k}-\Y_{\k}\|_{\D}\geq\nu)=0.
\end{align}
For $\bar{\mathbb{P}}^n_{\k}=\bar{\mathbb{P}}\circ {T^n_{\k}}^{-1}$, $n\in\N$ along with \eqref{ST3}-\eqref{ST4}, it follows that there exists $n_0\in\N$ such that 
\begin{align*}
	&\bar{\mathbb{P}}^{n_0}_{\k}(\omega:\|\mathcal{S}_{\bar{\mathbb{P}}^{n_0}_{\k}\circ\x_0^{-1}}(\x_0(\omega),\omega)-\Y_{\k}\|_{\mathcal{D}}<\nu)\nonumber\\&=\bar{\mathbb{P}}(\omega:\|\mathcal{S}_{\bar{\mathbb{P}}\circ\x_0^{-1}}(\x_0(\omega),T^{n_0}_{\k}(\omega))-\Y_{\k}\|_{\mathcal{D}}<\nu)>0.
\end{align*}
From \eqref{AS5}, we have that $\bar{\mathbb{P}}^{n_0}_{\k}\ll\bar{\mathbb{P}}$ which implies that 
\begin{align*}
	\mathbb{P}(\|\u-\Y_{\k}\|_{\mathcal{D}}<\nu)=\bar{\mathbb{P}}(\omega:\|\mathcal{S}_{\bar{\mathbb{P}}\circ\x_0^{-1}}(\x_0(\omega),\omega)-\Y_{\k}\|_{\mathcal{D}}<\nu)>0.
\end{align*}
Therefore 
\begin{align}\label{ST5}
	\text{supp}(\mathbb{P}\circ\u^{-1})\supset\bar{\mathcal{L}}.
\end{align}
Hence, \eqref{ST1} and \eqref{ST5} imply that $\text{supp}(\mathbb{P}\circ\u^{-1})=\bar{\mathcal{L}}$, as required.
\end{proof}
\section{Proof of Theorem \ref{4.1}}\label{sec5}\setcounter{equation}{0}
The proof of Theorem \ref{4.1} is based on a standard Galerkin approximation scheme. Let $\{f_1,f_2,\cdots,f_m,\cdots\}\subset\D(\A)$ be an orthonormal basis of $\H$ and set $\H_m:=\text{span}\{f_1,f_2,\cdots,f_m\}$. Let $\P_m:\V'\to\H_m$ be given by 
\begin{align}\label{S1}
	\P_m\h=\sum_{k=1}^{m}\langle\h,f_k\rangle f_k, \ \ \ \h\in\V'.
\end{align}
For $\h\in\H$, one can write $\P_m\h=\sum_{k=1}^{m}(\h,f_k) f_k$. For $\k\in\mathrm{L}^2(0,T;\mathrm{K})$, we  consider the following finite dimensional system in $\H_m$:
	\begin{equation}\label{C_SCBF_m}
	\left\{
	\begin{aligned}
		\d\mathrm{U}^{n,m}_{\k}(t)&=-\left[\mu\A\mathrm{U}^{n,m}_{\k}(t)+\B_m(\mathrm{U}^{n,m}_{\k}(t))+\alpha\mathrm{U}^{n,m}_{\k}(t)+\mathcal{C}_m(\mathrm{U}^{n,m}_{\k}(t))\right]\d t\\&\quad + \F_{1,m}(\mathrm{U}^{n,m}_{\k}(t))\Pi_m\d\W(t) +\F_{2,m}(\mathrm{U}^{n,m}_{\k}(t))\dot{\W}^n(t)\d t\\&\quad+\F_{3,m}(\mathrm{U}^{n,m}_{\k}(t))\k(t)\d t- \F_m(\mathrm{U}^{n,m}_{\k}(t))\d t,\\
		\mathrm{U}^{n,m}_{\k}(0)&=\P_m\x,
	\end{aligned}
	\right.
\end{equation}
where $\P_m$ and $\Pi_m$ are given in \eqref{S1} and Hypothesis \ref{HonG1}, respectively, $\B_m=\P_m\B$, $\mathcal{C}_m=\P_m\mathcal{C}$, $\F_m=\P_m\F$, $\F_{i,m}=\P_m\F_i$, $i=1,2,3.$ For any $t\in[0,T]$, $\mathbb{E}\left[\|\dot{\W}^n(t)\|^2_{\mathrm{K}}\right]=\frac{n}{\sigma}$. Also, from Hypothesis \ref{HonG}, we have for $\u_1,\u_2\in\H_m$
\begin{align}\label{S2}
	\langle[\F_2(\u_1)-\F_2(\u_2)]\dot{\W}^n(t),\u_1-\u_2\rangle &\leq \|\F_2(\u_1)-\F_2(\u_2)\|_{\mathcal{L}_2}\|\dot{\W}^n(t)\|_{\mathrm{K}}\|\u_1-\u_2\|_{\H}\nonumber\\&\leq\left[\rho(\u_2)+\|\dot{\W}^n(t)\|^2_{\mathrm{K}}\right]\|\u_1-\u_2\|^2_{\H}.
\end{align}
Similarly, one can show that 
\begin{align}\label{S3}
	\langle[\F_3(\u_1)-\F_3(\u_2)]\k(t),\u_1-\u_2\rangle \leq \left[\rho(\u_2)+\|\k(t)\|^2_{\mathrm{K}}\right]\|\u_1-\u_2\|^2_{\H}.
\end{align}
 It follows from Theorem \ref{LocMon}, Hypotheses \ref{HonG} and \ref{HonG1}, \eqref{S1} and \eqref{S2} that there exists a unique solution $\mathrm{U}^{n,m}_{\k}(\cdot)$ of the system \eqref{C_SCBF_m} (Theorem 3.1.1, \cite{LR}). 
 \vskip 2mm
 \noindent
 \textbf{A-priori energy estimates:}
 We provide an a-priori energy estimate of $\mathrm{U}^{n,m}_{\k}$ which will help us to get the solution for the system \eqref{C_SCBF}.
\begin{lemma}\label{5.1}
Under the assumptions in Theorem \ref{4.1}, we have
\begin{align}\label{5.5}
	&\sup_{n,m\in\N}\mathbb{E}\bigg[\sup_{t\in[0,T]}\|\mathrm{U}^{n,m}_{\k}(t)\|^{p}_{\H}+\int_{0}^{T}\|\mathrm{U}^{n,m}_{\k}(t)\|^{p-2}_{\H}\|\mathrm{U}^{n,m}_{\k}(t)\|^{2}_{\V}\d t\nonumber\\&+\int_{0}^{T}\|\mathrm{U}^{n,m}_{\k}(t)\|^{p-2}_{\H}\|\mathrm{U}^{n,m}_{\k}(t)
	\|^{r+1}_{\wi\L^{r+1}}\d t\bigg]\leq C e^{\int_{0}^{T}(1+\|\k(t)\|^2_{\mathrm{K}})\d t}\bigg(\mathbb{E}[\|\x\|^{p}_{\H}]+1\bigg),
\end{align}
where $C$ is a positive constant.
\end{lemma}
\begin{proof}
	Applying finite dimensional It\^o's formula to the process $\|\mathrm{U}^{n,m}_{\k}(t)\|^{p}_{\H}$, we obtain
	\begin{align}\label{S4}
		&\|\mathrm{U}^{n,m}_{\k}(t)\|^{p}_{\H}+p\int_{0}^{t}\|\mathrm{U}^{n,m}_{\k}(\zeta)\|^{p-2}\left(\mu\|\mathrm{U}^{n,m}_{\k}(\zeta)\|^2_{\V}+\alpha\|\mathrm{U}^{n,m}_{\k}(\zeta)\|^2_{\H}+\beta\|\mathrm{U}^{n,m}_{\k}(\zeta)\|^{r+1}_{\wi\L^{r+1}}\right)\d\zeta\nonumber\\&
		=\|\P_m\x\|^{p}_{\H}+\ p\int_{0}^{t}\|\mathrm{U}^{n,m}_{\k}(\zeta)\|^{p-2}\left(\mathrm{U}^{n,m}_{\k}(\zeta),\F_{1,m}(\mathrm{U}^{n,m}_{\k}(\zeta))\Pi_m\d\W(\zeta)\right)\d\zeta\nonumber\\&\quad+p\int_{0}^{t}\|\mathrm{U}^{n,m}_{\k}(\zeta)\|^{p-2}\left(\F_{2}(\mathrm{U}^{n,m}_{\k}(\zeta))\dot{\W}^n(\zeta),\mathrm{U}^{n,m}_{\k}(\zeta)\right)\d\zeta\nonumber\\&\quad+p\int_{0}^{t}\|\mathrm{U}^{n,m}_{\k}(\zeta)\|^{p-2}\left(\F_{3}(\mathrm{U}^{n,m}_{\k}(\zeta))\k(\zeta),\mathrm{U}^{n,m}_{\k}(\zeta)\right)\d\zeta\nonumber\\&\quad-p\int_{0}^{t}\|\mathrm{U}^{n,m}_{\k}(\zeta)\|^{p-2}\left(\F(\mathrm{U}^{n,m}_{\k}(\zeta))\k(\zeta),\mathrm{U}^{n,m}_{\k}(\zeta)\right)\d\zeta\nonumber\\&\quad+\frac{p}{2}\int_{0}^{t}\|\mathrm{U}^{n,m}_{\k}(\zeta)\|^{p-2}\|\F_{1,m}(\mathrm{U}^{n,m}_{\k}(\zeta))\Pi_m\|^2_{\mathcal{L}_2}\d\zeta\nonumber\\&\quad+p(p-2)\int_{0}^{t}\|\mathrm{U}^{n,m}_{\k}(\zeta)\|^{p-4}_{\H}\|(\F_{1,m}(\mathrm{U}^{n,m}_{\k}(\zeta))\Pi_m)^{*}\mathrm{U}^{n,m}_{\k}(\zeta)\|^2_{\H}\d\zeta\nonumber\\&\leq\|\x\|^{\p}_{\H}+\sum_{i=1}^{6}Q_{i}(t,n,m),
	\end{align}
where we have used $\|\P_m\x\|_{\H}\leq\|\x\|_{\H}$. Let us now estimate each term of \eqref{S4} separately.
\vskip 2mm
\noindent
\emph{Estimate for $Q_1(t,n,m)$:} Using BDG inequality, we have
\begin{align}\label{S5}
	\mathbb{E}\left[\sup_{t\in[0,T]}\left|Q_1(s,n,m)\right|\right]\leq\mathbb{E}\left[\frac{1}{6}\sup_{t\in[0,T]}\|\mathrm{U}^{n,m}_{\k}(t)\|^{p}_{\H}+C\int_{0}^{T}\left(\|\mathrm{U}^{n,m}_{\k}(\zeta)\|^{p}_{\H}\d\zeta+1\right)\d\zeta\right].
\end{align}
\vskip 2mm
\noindent
\emph{Estimate for $Q_2(t,n,m)$:} It is easy to see that $\dot{\W}^n(\zeta)=\frac{1}{\sigma}\int_{(\lfloor\frac{\zeta}{\sigma}\rfloor-1)\sigma}^{\lfloor\frac{\zeta}{\sigma}\rfloor\sigma}\Pi_n\d\W(\xi)$. Therefore, by Fubini's theorem, we have 
\begin{align*}
	&Q_2(t,n,m)\nonumber\\&=\int_{0}^{t}\sum_{j=0}^{\lfloor\frac{t}{\sigma}\rfloor}\frac{p}{\sigma}\int_{j\sigma}^{(j+1)\sigma\land t}1_{\{\xi\in[(j-1)\sigma\lor0,j\sigma]\}}\left(\|\mathrm{U}^{n,m}_{\k}(\zeta)\|^{p-2}_{\H}\F_2(\mathrm{U}^{n,m}_{\k}(\zeta))^{*}\mathrm{U}^{n,m}_{\k}(\zeta),\Pi_n\d\W(\xi)\right)\d\zeta.
\end{align*}
Using BDG inequality, Hypothesis \ref{HonG} and Lemma \ref{Young}, we get
\begin{align}\label{S6} 
	&\mathbb{E}\left[\sup_{t\in[0,T]}\left|Q_2(t,n,m)\right|\right]\nonumber\\&\leq\mathbb{E}\left[\int_{0}^{T}\sum_{j=0}^{2^n-1}\frac{p^2}{\sigma}\int_{j\sigma}^{(j+1)\sigma}\d\zeta 1_{\{\xi\in[(j-1)\sigma\lor0,j\sigma]\}}\|\mathrm{U}^{n,m}_{\k}(\zeta)\|^{2p-2}_{\H}\|\F_{2}(\mathrm{U}^{n,m}_{\k}(\zeta))\|^2_{\mathcal{L}_2}\d\xi\right]^{1/2}\nonumber\\&=\mathbb{E}\left[\int_{0}^{T}p^2 \|\mathrm{U}^{n,m}_{\k}(\zeta)\|^{2p-2}_{\H}\|\F_{2}(\mathrm{U}^{n,m}_{\k}(\zeta))\|^2_{\mathcal{L}_2}\d\zeta\right]^{1/2}\nonumber\\&\leq \mathbb{E}\left[C\int_{0}^{T}\|\mathrm{U}^{n,m}_{\k}(\zeta)\|^{2p-2}_{\H}\left(\|\mathrm{U}^{n,m}_{\k}(\zeta)\|^{2}_{\H}+1\right)\d\zeta\right]^{1/2}\nonumber\\&\leq\mathbb{E}\left[\frac{1}{6}\sup_{t\in[0,T]}\|\mathrm{U}^{n,m}_{\k}(t)\|^{p}_{\H}+C\left(\int_{0}^{T}\|\mathrm{U}^{n,m}_{\k}(\zeta)\|^{p}_{\H}\d\zeta+1\right)\right].
\end{align}
\vskip 2mm
\noindent
\emph{Estimate for $Q_3(t,n,m)$:} Using Hypothesis \ref{HonG} and Lemma \ref{Young}, we obtain
\begin{align}\label{S7}
	\mathbb{E}\left[\sup_{t\in[0,T]}\left|Q_3(t,n,m)\right|\right]&\leq \mathbb{E}\left[C\int_{0}^{T}\|\mathrm{U}^{n,m}_{\k}(\zeta)\|^{p-1}_{\H}\left(\|\mathrm{U}^{n,m}_{\k}(\zeta)\|_{\H}+1\right)\|\k(\zeta)\|_{\mathrm{K}}\d\zeta\right]\nonumber\\&\leq C\mathbb{E}\left[\int_{0}^{T}\left(\|\k(\zeta)\|^2_{\mathrm{K}}\|\mathrm{U}^{n,m}_{\k}(\zeta)\|^{p}_{\H}+\|\mathrm{U}^{n,m}_{\k}(\zeta)\|^{p}_{\H}+1\right)\d\zeta\right].
\end{align}
\vskip 2mm
\noindent
\emph{Estimate for $Q_4(t,n,m)$:} Making use of condition (H$'$.1) of Hypothesis \ref{HonG1}, we obtain
\begin{align}\label{S8}
	\mathbb{E}\left[\sup_{t\in[0,T]}\left|Q_4(s,n,m)\right|\right]\leq\mathbb{E}\left[\frac{1}{6}\sup_{t\in[0,T]}\|\mathrm{U}^{n,m}_{\k}(t)\|^{p}_{\H}+C\int_{0}^{T}\left(\|\mathrm{U}^{n,m}_{\k}(\zeta)\|^{p}_{\H}\d\zeta+1\right)\d\zeta\right].
\end{align}
\vskip 2mm
\noindent
\emph{Estimates for $Q_5(t,n,m)$ and $Q_6(t,n,m)$:} Using condition (H.1) of Hypothesis \ref{HonG} and then applying Lemma \ref{Young}, we get
\begin{align}\label{S9}
	\left|Q_5(t,n,m)+Q_6(t,n,m)\right|\leq C\int_{0}^{T}\left(\|\mathrm{U}^{n,m}_{\k}(\zeta)\|^{p}_{\H}+1\right)\d\zeta.
\end{align}
Taking supremum over $[0,T]$ and expectation on both sides of \eqref{S4}, then inserting \eqref{S5}-\eqref{S9} and applying Gronwall's inequality to the final estimate, we obtain \eqref{5.5}, which completes the proof.
\end{proof}
\vskip 2mm
\noindent
\textbf{Weak limits:} Lemma \ref{5.1} implies that 
\begin{align}
	\mathrm{U}^{n,m}_{\k}\in\mathrm{L}^{p}(\Omega; \mathrm{L}^{\infty}(0,T;\H))\cap\mathrm{L}^2(\Omega; \mathrm{L}^2(0,T;\V))\cap\mathrm{L}^{r+1}(\Omega;\mathrm{L}^{r+1}(0,T;\widetilde{\L}^{r+1})).
\end{align}
For $\M_m(\u) = \mu\A\u +\B_m(\u)+\alpha\u+\beta\mathcal{C}_m(\u),$ let us consider
\begin{align}\label{S10}
	&\E\left[\left|\int_0^T\langle\M_m(\mathrm{U}^{n,m}_{\k}(t)),\mathrm{V}(t)\rangle\d t\right|\right]\nonumber\\&\leq \mu\E\left[ \left|\int_0^T(\nabla\mathrm{U}^{n,m}_{\k}(t),\nabla\mathrm{V}(t))\d t\right|\right]+\E\left[\left|\int_0^T\langle \B(\mathrm{U}^{n,m}_{\k}(t),\mathrm{V}(t)),\mathrm{U}^{n,m}_{\k}(t)\rangle\d t\right|\right]\nonumber\\&\quad+\alpha\E\left[ \left|\int_0^T(\mathrm{U}^{n,m}_{\k}(t),\mathrm{V}(t))\d t\right|\right]+\beta\E\left[\left|\int_0^T\langle|\mathrm{U}^{n,m}_{\k}(t)|^{r-1}\mathrm{U}^{n,m}_{\k}(t),\mathrm{V}(t)\rangle\d t\right|\right]\nonumber\\&\leq C\E\left[ \int_0^T\|\mathrm{U}^{n,m}_{\k}(t)\|_{\V}\|\mathrm{V}(t)\|_{\V}\d t\right]+\E\left[\int_0^T\|\mathrm{U}^{n,m}_{\k}(t)\|_{\widetilde{\L}^{r+1}}\|\mathrm{U}^{n,m}_{\k}(t)\|_{\widetilde{\L}^{\frac{2(r+1)}{r-1}}}\|\mathrm{V}(t)\|_{\V}\d t\right]\nonumber\\&\quad+\beta\E\left[\int_0^T\|\mathrm{U}^{n,m}_{\k}(t)\|_{\widetilde{\L}^{r+1}}^{r}\|\mathrm{V}(t)\|_{\widetilde{\L}^{r+1}}\d t\right]\nonumber\\&\leq C \mathbb{E}\left[\left(\int_0^T\|\mathrm{U}^{n,m}_{\k}(t)\|_{\V}^2\d t\right)^{1/2}\left(\int_0^T\|\mathrm{V}(t)\|_{\V}^2\d t\right)^{1/2}\right]\nonumber\\&\quad+\E\left[\left(\int_0^T\|\mathrm{U}^{n,m}_{\k}(t)\|_{\widetilde{\L}^{r+1}}^{r+1}\d t\right)^{\frac{1}{r-1}}\left(\int_0^T\|\mathrm{U}^{n,m}_{\k}(t)\|_{\H}^2\d t\right)^{\frac{r-3}{2(r-1)}}\left(\int_0^T\|\mathrm{V}(t)\|_{\V}^2\d t\right)^{1/2}\right]\nonumber\\&\quad+\beta\E\left[\left(\int_0^T\|\mathrm{U}^{n,m}_{\k}(t)\|_{\widetilde{\L}^{r+1}}^{r+1}\d t\right)^{\frac{r}{r+1}}\left(\int_0^T\|\mathrm{V}(t)\|_{\widetilde{\L}^{r+1}}^{r+1}\d t\right)^{\frac{1}{r+1}}\right]\nonumber\\&\leq C \left\{\E\left(\int_0^T\|\mathrm{U}^{n,m}_{\k}(t)\|_{\V}^2\d t\right)\right\}^{1/2}\left\{\E\left(\int_0^T\|\mathrm{V}(t)\|_{\V}^2\d t\right)\right\}^{1/2}\nonumber\\&\quad+T^{\frac{r-3}{2(r-1)}}\left\{\E\left(\int_0^T\|\mathrm{U}^{n,m}_{\k}(t)\|_{\widetilde{\L}^{r+1}}^{r+1}\d t\right)\right\}^{\frac{1}{r-1}}\left\{\E\left(\sup_{t\in[0,T]}\|\mathrm{U}^{n,m}_{\k}(t)\|_{\H}^2\right)\right\}^{\frac{r-3}{2(r-1)}}\nonumber\\&\quad\times\left\{\E\left(\int_0^T\|\mathrm{V}(t)\|_{\V}^2\d t\right)\right\}^{\frac{1}{2}}+\beta\left\{\E\left(\int_0^T\|\mathrm{U}^{n,m}_{\k}(t)\|_{\widetilde{\L}^{r+1}}^{r+1}\d t\right)\right\}^{\frac{r}{r+1}}\nonumber\\&\quad\times\left\{\E\left(\int_0^T\|\mathrm{V}(t)\|_{\widetilde{\L}^{r+1}}^{r+1}\d t\right)\right\}^{\frac{1}{r+1}},
\end{align}
for $r>3$ and for all $\mathrm{V}\in\mathrm{L}^2(\Omega;\mathrm{L}^2(0,T;\V))\cap\mathrm{L}^{r+1}(\Omega;\mathrm{L}^{r+1}(0,T;\widetilde{\L}^{r+1}))$, where we have used Lemmas \ref{Holder} and \ref{Interpolation}. Similarly for $r=3$, we obtain 
\begin{align}\label{S10*}
	&\E\left[\left|\int_0^T\langle\M_m(\mathrm{U}^{n,m}_{\k}(t)),\mathrm{V}(t)\rangle\d t\right|\right]\nonumber\\&\leq  C \left\{\E\left(\int_0^T\|\mathrm{U}^{n,m}_{\k}(t)\|_{\V}^2\d t\right)\right\}^{1/2}\left\{\E\left(\int_0^T\|\mathrm{V}(t)\|_{\V}^2\d t\right)\right\}^{1/2}\nonumber\\&\quad+\left\{\E\left(\int_0^T\|\mathrm{U}^{n,m}_{\k}(t)\|_{\widetilde{\L}^{4}}^{4}\d t\right)\right\}^{\frac{1}{2}}\left\{\E\left(\int_0^T\|\mathrm{V}(t)\|_{\V}^2\d t\right)\right\}^{\frac{1}{2}}\nonumber\\&\quad+\beta\left\{\E\left(\int_0^T\|\mathrm{U}^{n,m}_{\k}(t)\|_{\widetilde{\L}^{4}}^{4}\d t\right)\right\}^{\frac{3}{4}}\left\{\E\left(\int_0^T\|\mathrm{V}(t)\|_{\widetilde{\L}^{4}}^{4}\d t\right)\right\}^{\frac{1}{4}},
\end{align}
for all $\mathrm{V}\in\mathrm{L}^2(\Omega;\mathrm{L}^2(0,T;\V))\cap\mathrm{L}^{4}(\Omega;\mathrm{L}^{4}(0,T;\widetilde{\L}^{4}))$. The above inequalities \eqref{S10} and \eqref{S10*} along with Lemma \ref{5.1} imply that for all $r\geq3$,
\begin{align}\label{S11}
	\M_m(\mathrm{U}^{n,m}_{\k})\in\mathrm{L}^2(\Omega;\mathrm{L}^2(0,T ;\V'))+\mathrm{L}^{\frac{r+1}{r}}(\Omega;\mathrm{L}^{\frac{r+1}{r}}(0,T;\widetilde{\L}^{\frac{r+1}{r}})).	
\end{align}
By condition (H.1) of Hypothesis \ref{HonG} and Lemma \ref{Young}, we get 
\begin{align}\label{S12}
	\|\F_{2,m}(\mathrm{U}^{n,m}_{\k}(t))\dot{\W}^n(t)\|^{2}_{\H}&\leq\|\F_2(\mathrm{U}^{n,m}_{\k}(t))\|^{2}_{\mathcal{L}_2}\|\dot{\W}^n(t)\|^{2}_{\mathrm{K}}\nonumber\\&\leq L_1\left(1+\|\mathrm{U}^{n,m}_{\k}(t)\|^2_{\H}\right) \|\dot{\W}^n(t)\|^{2}_{\mathrm{K}}\nonumber\\&\leq C \left(1+\|\mathrm{U}^{n,m}_{\k}(t)\|^{p}_{\H}+\|\dot{\W}^n(t)\|^{\frac{2p}{p-2}}_{\mathrm{K}}\right).
\end{align}
From \eqref{AS1}, we infer that $\|\dot{\W}^n(j\sigma)\|_{\mathrm{K}}$ (for $j=1,2,\ldots,2^n$) are independent centered normal random variables with $\mathbb{E}\left[\|\dot{\W}^n(j\sigma)\|^{2}_{\mathrm{K}}\right]=\frac{n}{\sigma},$ which gives 
\begin{align}\label{S13}
	\mathbb{E}\left[\int_{0}^{T}\|\dot{\W}^n(t)\|^{\frac{2p}{p-2}}_{\mathrm{K}}\d t\right]&=\sum_{j=1}^{2^n}\sigma\mathbb{E}\left[\|\dot{\W}^n(j\sigma)\|^{\frac{2p}{p-2}}_{\mathrm{K}}\right]\nonumber\\&\leq C\sigma\sum_{j=1}^{2^n}\left(\mathbb{E}\left[\|\dot{\W}^n(j\sigma)\|^{2}_{\mathrm{K}}\right]\right)^{\frac{p}{p-2}}=C\sigma2^n \left(\frac{n}{\sigma}\right)^{\frac{p}{p-2}}.
\end{align}
From Lemma \ref{5.5}, \eqref{S12}-\eqref{S13} and continuous embedding $\H\hookrightarrow\V'$, we obtain that for each $n\in\N$
\begin{align}\label{S14}
	\F_{2,m}(\mathrm{U}^{n,m}_{\k})\dot{\W}^n\in\mathrm{L}^2(\Omega;\mathrm{L}^2(0,T ;\V')) \text{ uniformly for } m\in\N.
\end{align}
For $\k\in\mathrm{L}^2(0,T;\mathrm{K})$, from condition (H.1) of Hypothesis \ref{HonG}, Lemma \ref{Young} and continuous embedding $\H\hookrightarrow\V'$, we obtain
\begin{align}\label{S15}
	\|\F_{3,m}(\mathrm{U}^{n,m}_{\k}(t))\k(t)\|^{2}_{\H}&\leq\|\F_3(\mathrm{U}^{n,m}_{\k}(t))\|^{2}_{\mathcal{L}_2}\|\k(t)\|^{2}_{\mathrm{K}}\nonumber\\&\leq L_1\left(1+\|\mathrm{U}^{n,m}_{\k}(t)\|^2_{\H}\right) \|\k(t)\|^{2}_{\mathrm{K}}\nonumber\\&\leq C\left(1+\|\mathrm{U}^{n,m}_{\k}(t)\|^{p}_{\H}\right) \|\k(t)\|^{2}_{\mathrm{K}},
\end{align}
which gives that for each $n\in\N$
\begin{align}\label{S16}
	\F_{3,m}(\mathrm{U}^{n,m}_{\k})\k\in\mathrm{L}^2(\Omega;\mathrm{L}^2(0,T ;\V')) \text{ uniformly for } m\in\N.
\end{align}
Similarly, from from condition (H$'$.1) of Hypothesis \ref{HonG1}, Lemma \ref{Young} and continuous embedding $\H\hookrightarrow\V'$, we obtain that for each $n\in\N$
\begin{align}\label{S17}
	\F_{m}(\mathrm{U}^{n,m}_{\k})\in\mathrm{L}^2(\Omega;\mathrm{L}^2(0,T ;\V')) \text{ uniformly for } m\in\N.
\end{align}
Therefore, for each $n\in\N$, we can find a subsequence $m_j$ (depends on $n$) such that $m_j\to\infty$ as $j\to\infty$ and have the following convergence in hand as $j\to\infty$:
\begin{equation}\label{S18}
	\left\{
	\begin{aligned}
		\mathrm{U}^{n,m_j}_{\k}&\xrightharpoonup{w^*} \widetilde{\mathrm{U}}^n_{\k} \text{ in } \mathrm{L}^{p}(\Omega;\mathrm{L}^{\infty}(0,T ;\H)),\\
		\mathrm{U}^{n,m_j}_{\k}&\xrightharpoonup{w} \widetilde{\mathrm{U}}^n_{\k} \text{ in } \mathrm{L}^2(\Omega;\mathrm{L}^2(0,T ;\V)),\\
		\M_{m_j} (\mathrm{U}^{n,m_j}_{\k})&\xrightharpoonup{w} \M^{n}_{0} \text{ in } \mathrm{L}^2(\Omega;\mathrm{L}^2(0,T ;\V'))+\mathrm{L}^{\frac{r+1}{r}}(\Omega;\mathrm{L}^{\frac{r+1}{r}}(0,T;\widetilde{\L}^{\frac{r+1}{r}})),\\
		\F_{2,m_j}(\mathrm{U}^{n,m_j}_{\k})\dot{\W}^n&\xrightharpoonup{w} \mathrm{N}^{n}_{2,\k} \text{ in } \mathrm{L}^2(\Omega;\mathrm{L}^2(0,T ;\V')),\\
		\F_{3,m_j}(\mathrm{U}^{n,m_j}_{\k})\k&\xrightharpoonup{w} \mathrm{N}^{n}_{3,\k} \text{ in } \mathrm{L}^2(\Omega;\mathrm{L}^2(0,T ;\V')),\\
		\F_{m_j}(\mathrm{U}^{n,m_j}_{\k})&\xrightharpoonup{w} \mathrm{N}^{n}_{\k} \text{ in } \mathrm{L}^2(\Omega;\mathrm{L}^2(0,T ;\V')),\\
		\F_{1,m_j}(\mathrm{U}^{n,m_j}_{\k})&\xrightharpoonup{w} \mathrm{N}^{n}_{1,\k} \text{ in } \mathrm{L}^2(\Omega;\mathrm{L}^2(0,T ;\mathcal{L}_2(\mathrm{K},\H))).
	\end{aligned}
	\right.
\end{equation}
The final convergence in \eqref{S18} implies that
\begin{align}\label{S19}
	\int_{0}^{\cdot}\F_{1,m_j}(\mathrm{U}^{n,m_j}_{\k})\Pi_{m_j}\d\W\to\int_{0}^{\cdot}\mathrm{N}^{n}_{1,\k}\d\W \ \text{ weakly in }\  \mathrm{L}^2(\Omega,\mathrm{L}^{\infty}(0,T,\H)).
\end{align}
\vskip 2mm
\noindent
\textbf{It\^o stochastic differential satisfied by $\mathrm{U}^n_{\k}(\cdot)$:} Now following the method given in \cite{LR,MTM2} and using the convergences obtained in \eqref{S18}, we prove that $\mathrm{U}^n_{\k}(\cdot)$  satisfies the system \eqref{C_SCBF}, where $\mathrm{U}^n_{\k}$ is defined by 
\begin{align}\label{S20}
	\mathrm{U}^n_{\k}(t):=\x+\int_{0}^{t}\left\{-\M^{n}_{0}(\zeta)+\mathrm{N}^n_{2,\k}(\zeta)+\mathrm{N}^n_{3,\k}(\zeta)-\mathrm{N}^n_{\k}(\zeta)\right\}\d\zeta+\int_{0}^{t}\mathrm{N}^n_{1,\k}(\zeta)\d\W(\zeta),
\end{align}
in $\V'+\mathrm{L}^{\frac{r+1}{r}}$. Our next aim is to prove that $\mathrm{U}^n_{\k}(t)=\widetilde{\mathrm{U}}^n_{\k}(t)$, $\d t\otimes\mathbb{P}$-a.s. From \eqref{C_SCBF_m} and Fubini's theorem, for all $\upsilon\in\cup_{m\geq1}\H_m\subset\V$, $\varphi\in\mathrm{L}^{\infty}(\Omega;\mathrm{L}^{\infty}(0,T))$, we obtain
\begin{align*}
&\mathbb{E}\left[\int_{0}^{T}\left\langle\widetilde{\mathrm{U}}^n_{\k}(t),\varphi(t)\upsilon\right\rangle\d t\right]\nonumber\\&=\lim_{j\to\infty}\mathbb{E}\left[\int_{0}^{T}\left\langle\mathrm{U}^{n,m_j}_{\k}(t),\varphi(t)\upsilon\right\rangle\d t\right]\nonumber\\&=\lim_{j\to\infty}\mathbb{E}\bigg[\int_{0}^{T}\left(\mathrm{U}^{n,m_j}_{\k}(0),\varphi(t)\upsilon\right)\d t -\int_{0}^{T}\int_{0}^{t}\left(\M_{m_j}(\mathrm{U}^{n,m_j}_{\k}(\zeta)),\varphi(t)\upsilon\right)\d\zeta\d t\nonumber\\&\quad+ \int_{0}^{T}\int_{0}^{t}\left(\F_{2,m_j}(\mathrm{U}^{n,m_j}_{\k}(\zeta)),\varphi(t)\upsilon\right)\d\zeta\d t+\int_{0}^{T}\int_{0}^{t}\left(\F_{3,m_j}(\mathrm{U}^{n,m_j}_{\k}(\zeta)),\varphi(t)\upsilon\right)\d\zeta\d t\nonumber\\&\quad-\int_{0}^{T}\int_{0}^{t}\left(\F_{m_j}(\mathrm{U}^{n,m_j}_{\k}(\zeta)),\varphi(t)\upsilon\right)\d\zeta\d t+\int_{0}^{T}\left(\int_{0}^{t}\F_{1,m_j}(\mathrm{U}^{n,m_j}_{\k}(\zeta))\Pi_{m_j}\d\W(\zeta),\varphi(t)\upsilon\right)\d t\bigg]\nonumber\\&=\lim_{j\to\infty}\mathbb{E}\left[\left(\mathrm{U}^{n,m_j}(0),\upsilon\right)\int_{0}^{T}\varphi(t)\d t\right]-\lim_{j\to\infty}\mathbb{E}\left[\int_{0}^{T}\left(\M_{m_j}(\mathrm{U}^{n,m_j}_{\k}(\zeta)), \upsilon\int_{\zeta}^{T}\varphi(t)\d t \right)\d\zeta\right]\nonumber\\&\quad+\lim_{j\to\infty}\mathbb{E}\left[\int_{0}^{T}\left(\F_{2,m_j}(\mathrm{U}^{n,m_j}_{\k}(\zeta))+\F_{3,m_j}(\mathrm{U}^{n,m_j}_{\k}(\zeta))-\F_{m_j}(\mathrm{U}^{n,m_j}_{\k}(\zeta)), \upsilon\int_{\zeta}^{T}\varphi(t)\d t \right)\d\zeta\right]\nonumber\\&\quad+\lim_{j\to\infty}\mathbb{E}\left[\int_{0}^{T}\left(\int_{0}^{t}\F_{1,m_j}(\mathrm{U}^{n,m_j}_{\k}(\zeta))\Pi_{m_j}\d\W(\zeta),\varphi(t)\upsilon\right)\d t\right]\nonumber\\=&\mathbb{E}\left[\int_{0}^{T}\left\langle\x+\int_{0}^{t}\left\{-\M^{n}_{0}(\zeta)+\mathrm{N}^n_{2,\k}(\zeta)+\mathrm{N}^n_{3,\k}(\zeta)-\mathrm{N}^n_{\k}(\zeta)\right\}\d\zeta+\int_{0}^{t}\mathrm{N}^n_{1,\k}(\zeta)\d\W(\zeta), \varphi(t)\upsilon\right\rangle\d t\right], 
\end{align*}
which implies that 
\begin{align*}
	\widetilde{\mathrm{U}}^n_{\k}(t):=\x+\int_{0}^{t}\left\{-\M^{n}_{0}(\zeta)+\mathrm{N}^n_{2,\k}(\zeta)+\mathrm{N}^n_{3,\k}(\zeta)-\mathrm{N}^n_{\k}(\zeta)\right\}\d\zeta+\int_{0}^{t}\mathrm{N}^n_{1,\k}(\zeta)\d\W(\zeta), 
\end{align*}
in $\V'+\wi\L^{\frac{r+1}{r}}$ for all $t\in[0,T]$. Hence $\mathrm{U}^n_{\k}(t)=\widetilde{\mathrm{U}}^n_{\k}(t)$, $\d t\otimes\mathbb{P}$-a.s. Further, Lemma \ref{5.1} along with Theorem 4.2.5 from \cite{LR} (one can prove energy equality satisfied by $\mathrm{U}^n_{\k}(\cdot)$  using same method used in \cite{MTM2}, where the author obtained the energy equality with the help of eigenfunctions of the Stokes operator and a construction available in \cite{GGP}) implies that $\mathrm{U}^n_{\k}$ is an $\H$-valued continuous $(\mathscr{F}_{t})$-adapted process. Therefore, for the existence of solutions to the system \eqref{C_SCBF}, it only remains to prove that
\begin{align}
	-\left[\mu\A\mathrm{U}^n_{\k}+\B(\mathrm{U}^n_{\k})+\alpha\mathrm{U}^n_{\k}+\mathcal{C}(\mathrm{U}^n_{\k})\right]& +\F_2(\mathrm{U}^n_{\k})\dot{\W}^n+\F_3(\mathrm{U}^n_{\k})\k- \F(\mathrm{U}^n_{\k})\nonumber\\&=-\M^{n}_{0}+\mathrm{N}^n_{2,\k}+\mathrm{N}^n_{3,\k}-\mathrm{N}^n_{\k},\label{S21}\\
	\F_1(\mathrm{U}^n_{\k})&=\mathrm{N}^n_{1,\k}, \ \ \ \ \ \d t\otimes\mathbb{P}.\label{S22}
\end{align}
Let $\Phi$ and $\rho$ be defined by \eqref{MP36} and Hypothesis \ref{HonG}, respectively and set 
\begin{align*}
	\mathcal{M}:=\biggl\{\phi:\phi \text{ is }& \V\cap\widetilde{\L}^{r+1}\text{-valued } \mathscr{F}_t\text{-adapted process,}\ \ \\& \mathbb{E}\left[\int_{0}^{T}\Phi(\phi(\zeta))\d\zeta\right]<\infty \text{ and } \mathbb{E}\left[\int_{0}^{T}\rho(\phi(\zeta))\d\zeta\right]<\infty\biggr\}.
\end{align*}
For $\phi\in\mathrm{L}^{p}(\Omega; \mathrm{L}^{\infty}(0,T;\H))\cap\mathrm{L}^2(\Omega; \mathrm{L}^2(0,T;\V))\cap\mathrm{L}^{r+1}(\Omega;\mathrm{L}^{r+1}(0,T;\widetilde{\L}^{r+1}))\cap\mathcal{M}$, an application of finite dimensional It\^o's formula yields 
\begin{align*}
&	\mathbb{E}\bigg[e^{-\int_{0}^{t}\left(2\Phi(\phi(\zeta))+5\rho(\phi(\zeta))+\|\dot{\W}^n(\zeta)\|^2_{\mathrm{K}}+\|\k(\zeta)\|^2_{\mathrm{K}}\right)\d\zeta}\|\mathrm{U}^{n,m_j}_{\k}(t)\|^2_{\H}\bigg]-\mathbb{E}\bigg[\|\mathrm{P}_m\x\|^2_{\H}\bigg]\nonumber\\&=\mathbb{E}\bigg[\int_{0}^{t}e^{-\int_{0}^{s}\left(2\Phi(\phi(\zeta))+5\rho(\phi(\zeta))+\|\dot{\W}^n(\zeta)\|^2_{\mathrm{K}}+\|\k(\zeta)\|^2_{\mathrm{K}}\right)\d\zeta}\bigg\{2\big\langle\M(\mathrm{U}^{n,m_j}_{\k}(s)),\mathrm{U}^{n,m_j}_{\k}(s)\big\rangle\nonumber\\&\qquad+2\big(\F_{2}(\mathrm{U}^{n,m_j}_{\k}(s))\dot{\W}^n(s)+\F_{3}(\mathrm{U}^{n,m_j}_{\k}(s))\k(s)-\F(\mathrm{U}^{n,m_j}_{\k}(s)),\mathrm{U}^{n,m_j}_{\k}(s)\big)\nonumber\\&\qquad+\|\F_{1,m_j}(\mathrm{U}^{n,m_j}_{\k}(s))\Pi_{m_j}\|^2_{\mathcal{L}_2}\nonumber\\&\qquad-\big(2\Phi(\phi(s))+5\rho(\phi(s))+\|\dot{\W}^n(s)\|^2_{\mathrm{K}}+\|\k(s)\|^2_{\mathrm{K}}\big)\|\mathrm{U}^{n,m_j}_{\k}(s)\|^2_{\H}\bigg\}\d s\bigg]\nonumber\\&\leq \mathbb{E}\bigg[\int_{0}^{t}e^{-\int_{0}^{s}\left(2\Phi(\phi(\zeta))+5\rho(\phi(\zeta))+\|\dot{\W}^n(\zeta)\|^2_{\mathrm{K}}+\|\k(\zeta)\|^2_{\mathrm{K}}\right)\d\zeta}\nonumber\\&\qquad\times\bigg\{2\big\langle\M(\mathrm{U}^{n,m_j}_{\k}(s))-\M(\phi(s)),\mathrm{U}^{n,m_j}_{\k}(s)-\phi(s)\big\rangle+\|\F_{1}(\mathrm{U}^{n,m_j}_{\k}(s))-\F_1(\phi(s))\|^2_{\mathcal{L}_2}\nonumber\\&\qquad+2\big(\left[\F_{2}(\mathrm{U}^{n,m_j}_{\k}(s))-\F_2(\phi(s))\right]\dot{\W}^n(s)+\left[\F_{3}(\mathrm{U}^{n,m_j}_{\k}(s))-\F_3(\phi(s))\right]\k(s)\nonumber\\&\qquad-\left[\F(\mathrm{U}^{n,m_j}_{\k}(s))-\F(\phi(s))\right],\mathrm{U}^{n,m_j}_{\k}(s)-\phi(s)\big)\nonumber\\&\qquad-\big(2\Phi(\phi(s))+5\rho(\phi(s))+\|\dot{\W}^n(s)\|^2_{\mathrm{K}}+\|\k(s)\|^2_{\mathrm{K}}\big)\|\mathrm{U}^{n,m_j}_{\k}(s)-\phi(s)\|^2_{\H}\bigg\}\d s\bigg]\nonumber\\&
+\mathbb{E}\bigg[\int_{0}^{t}e^{-\int_{0}^{s}\left(2\Phi(\phi(\zeta))+5\rho(\phi(\zeta))+\|\dot{\W}^n(\zeta)\|^2_{\mathrm{K}}+\|\k(\zeta)\|^2_{\mathrm{K}}\right)\d\zeta}\nonumber\\&\qquad\times\bigg\{2\big\langle\M(\mathrm{U}^{n,m_j}_{\k}(s))-\M(\phi(s)),\phi(s)\big\rangle+2\big\langle\M(\phi(s)),\mathrm{U}^{n,m_j}_{\k}(s)\big\rangle\nonumber\\&\qquad+2\big(\left[\F_{2}(\mathrm{U}^{n,m_j}_{\k}(s))-\F_2(\phi(s))\right]\dot{\W}^n(s),\phi(s)\big)+2\big(\F_{2}(\phi(s))\dot{\W}^n(s),\mathrm{U}^{n,m_j}_{\k}(s)\big)\nonumber\\&\qquad+2\big(\left[\F_{3}(\mathrm{U}^{n,m_j}_{\k}(s))-\F_3(\phi(s))\right]\k(s),\phi(s)\big)+2\big(\F_3(\phi(s))\k(s),\mathrm{U}^{n,m_j}_{\k}(s)\big)\nonumber\\&\qquad-2\big(\F(\mathrm{U}^{n,m_j}_{\k}(s))-\F(\phi(s)),\phi(s)\big)-2\big(\F(\phi(s)),\mathrm{U}^{n,m_j}_{\k}(s)\big)\nonumber\\&\qquad-\|\F_{1}(\phi(s))\|^2_{\mathcal{L}_2}+2\big\langle\F_1(\mathrm{U}^{n,m_j}_{\k}(s)),\F_1(\phi(s))\big\rangle_{\mathcal{L}_2}\nonumber\\&\qquad-2\big(2\Phi(\phi(s))+5\rho(\phi(s))+\|\dot{\W}^n(s)\|^2_{\mathrm{K}}+\|\k(s)\|^2_{\mathrm{K}}\big)\big(\mathrm{U}^{n,m_j}_{\k}(s),\phi(s)\big)\nonumber\\&\qquad+\big(2\Phi(\phi(s))+5\rho(\phi(s))+\|\dot{\W}^n(s)\|^2_{\mathrm{K}}+\|\k(s)\|^2_{\mathrm{K}}\big)\|\phi(s)\|^2_{\H}\bigg\}\d s\bigg].
\end{align*}
For any non-negative $\psi\in\mathrm{L}^{\infty}(0,T)$, we have
\begin{align*}
	&\mathbb{E}\bigg[\int_{0}^{T}\psi(t)\|\mathrm{U}^n_{\k}(t)\|^2_{\H}\d t\bigg]\\&=\lim_{j\to\infty}	\mathbb{E}\bigg[\int_{0}^{T}\psi(t)\big\langle\mathrm{U}^n_{\k}(t),\mathrm{U}^{n,m_j}_{\k}(t)\big\rangle\d t\bigg]\\&\leq	\left(\mathbb{E}\bigg[\int_{0}^{T}\psi(t)\|\mathrm{U}^n_{\k}(t)\|^2_{\H}\d t\bigg]\right)^{1/2}\liminf_{j\to\infty}	\left(\mathbb{E}\bigg[\int_{0}^{T}\psi(t)\|\mathrm{U}^{n,m_j}_{\k}(t)\|^2_{\H}\d t\bigg]\right)^{1/2}.
\end{align*}
With the help of Theorem \ref{LocMon}, Hypotheses \ref{HonG} and \ref{HonG1}, we obtain 
\begin{align}\label{S23}
	&\mathbb{E}\bigg[\int_{0}^{T}\psi(t)\left\{e^{-\int_{0}^{t}\left(2\Phi(\phi(\zeta))+5\rho(\phi(\zeta))+\|\dot{\W}^n(\zeta)\|^2_{\mathrm{K}}+\|\k(\zeta)\|^2_{\mathrm{K}}\right)\d\zeta}\|\mathrm{U}^{n}_{\k}(t)\|^2_{\H}-\|\x\|^2_{\H}\right\}\d t\bigg]\nonumber\\&\leq	\liminf_{j\to\infty}\mathbb{E}\bigg[\int_{0}^{T}\psi(t)\left\{e^{-\int_{0}^{t}\left(2\Phi(\phi(\zeta))+5\rho(\phi(\zeta))+\|\dot{\W}^n(\zeta)\|^2_{\mathrm{K}}+\|\k(\zeta)\|^2_{\mathrm{K}}\right)\d\zeta}\|\mathrm{U}^{n,m_j}_{\k}(t)\|^2_{\H}-\|\mathrm{P}_m\x\|^2_{\H}\right\}\d t\bigg]\nonumber\\&\leq\mathbb{E}\bigg[\int_{0}^{T}\psi(t)\bigg(\int_{0}^{t}e^{-\int_{0}^{s}\left(2\Phi(\phi(\zeta))+5\rho(\phi(\zeta))+\|\dot{\W}^n(\zeta)\|^2_{\mathrm{K}}+\|\k(\zeta)\|^2_{\mathrm{K}}\right)\d\zeta}\nonumber\\&\qquad\times\bigg\{2\big\langle\M_0^n(s)-\M(\phi(s)),\phi(s)\big\rangle+2\big\langle\M(\phi(s)),\mathrm{U}^{n}_{\k}(s)\big\rangle\nonumber\\&\qquad+2\big\langle\mathrm{N}_{2,\k}^n(s)-\F_2(\phi(s))\dot{\W}^n(s),\phi(s)\big\rangle+2\big\langle\F_{2}(\phi(s))\dot{\W}^n(s),\mathrm{U}^{n}_{\k}(s)\big\rangle\nonumber\\&\qquad+2\big\langle\mathrm{N}_{3,\k}^n(s)-\F_3(\phi(s))\k(s),\phi(s)\big\rangle+2\big\langle\F_3(\phi(s))\k(s),\mathrm{U}^{n}_{\k}(s)\big\rangle\nonumber\\&\qquad-2\big\langle\mathrm{N}_{\k}^n(s)-\F(\phi(s)),\phi(s)\big\rangle-2\big\langle\F(\phi(s)),\mathrm{U}^{n}_{\k}(s)\big\rangle\nonumber\\&\qquad-\|\F_{1}(\phi(s))\|^2_{\mathcal{L}_2}+2\big\langle\mathrm{N}_{1,\k}^n(s),\F_1(\phi(s))\big\rangle_{\mathcal{L}_2}\nonumber\\&\qquad-2\big(2\Phi(\phi(s))+5\rho(\phi(s))+\|\dot{\W}^n(s)\|^2_{\mathrm{K}}+\|\k(s)\|^2_{\mathrm{K}}\big)\big\langle\mathrm{U}^{n}_{\k}(s),\phi(s)\big\rangle\nonumber\\&\qquad+\big(2\Phi(\phi(s))+5\rho(\phi(s))+\|\dot{\W}^n(s)\|^2_{\mathrm{K}}+\|\k(s)\|^2_{\mathrm{K}}\big)\|\phi(s)\|^2_{\H}\bigg\}\d s\bigg)\d t\bigg].
\end{align}
From \eqref{S20} using the energy equality established in \cite{MTM2}, we also have
\begin{align*}
	&\mathbb{E}\bigg[e^{-\int_{0}^{t}\left(2\Phi(\phi(\zeta))+5\rho(\phi(\zeta))+\|\dot{\W}^n(\zeta)\|^2_{\mathrm{K}}+\|\k(\zeta)\|^2_{\mathrm{K}}\right)\d\zeta}\|\mathrm{U}^{n}_{\k}(t)\|^2_{\H}-\|\x\|^2_{\H}\bigg]\nonumber\\&=\mathbb{E}\bigg[\int_{0}^{t}e^{-\int_{0}^{s}\left(2\Phi(\phi(\zeta))+5\rho(\phi(\zeta))+\|\dot{\W}^n(\zeta)\|^2_{\mathrm{K}}+\|\k(\zeta)\|^2_{\mathrm{K}}\right)\d\zeta}\nonumber\\&\qquad\times\bigg\{2\big\langle\M_0^n(s)+\mathrm{N}^n_{1,\k}(s)+\mathrm{N}^n_{2,\k}(s)-\mathrm{N}^n_{\k}(s),\mathrm{U}^n_{\k}(s)\big\rangle+\|\mathrm{N}^n_{1,\k}(s)\|^2_{\mathcal{L}_2}\nonumber\\&\qquad-\left(2\Phi(\phi(s))+5\rho(\phi(s))+\|\dot{\W}^n(s)\|^2_{\mathrm{K}}+\|\k(s)\|^2_{\mathrm{K}}\right)\|\mathrm{U}^{n}_{\k}(s)\|^2_{\H}\bigg\}\d s\bigg],
\end{align*}
which together with \eqref{S23} gives
\begin{align}\label{S24}
	0&\geq\mathbb{E}\bigg[\int_{0}^{T}\psi(t)\bigg(\int_{0}^{t}e^{-\int_{0}^{s}\left(2\Phi(\phi(\zeta))+5\rho(\phi(\zeta))+\|\dot{\W}^n(\zeta)\|^2_{\mathrm{K}}+\|\k(\zeta)\|^2_{\mathrm{K}}\right)\d\zeta}\nonumber\\&\qquad\times\bigg\{2\big\langle\M_0^n(s)-\M(\phi(s)),\mathrm{U}^n_{\k}(s)-\phi(s)\big\rangle+2\big\langle\mathrm{N}_{2,\k}^n(s)-\F_2(\phi(s))\dot{\W}^n(s)+\mathrm{N}_{3,\k}^n(s)\nonumber\\&\qquad-\F_3(\phi(s))\k(s)-\mathrm{N}_{\k}^n(s)+\F(\phi(s)),\mathrm{U}^n_{\k}(s)-\phi(s)\big\rangle+\|\F_{1}(\phi(s))-\mathrm{N}^n_{1,k}(s)\|^2_{\mathcal{L}_2}\nonumber\\&\qquad-\big(2\Phi(\phi(s))+5\rho(\phi(s))+\|\dot{\W}^n(s)\|^2_{\mathrm{K}}+\|\k(s)\|^2_{\mathrm{K}}\big)\|\mathrm{U}^n_{\k}(s)-\phi(s)\|^2_{\H}\bigg\}\d s\bigg)\d t\bigg].
\end{align}
Lemma \ref{5.1} and Hypothesis \ref{HonG} imply that 
$$\mathrm{U}^n_{\k}\in\mathrm{L}^{p}(\Omega; \mathrm{L}^{\infty}(0,T;\H))\cap\mathrm{L}^2(\Omega; \mathrm{L}^2(0,T;\V))\cap\mathrm{L}^{r+1}(\Omega;\mathrm{L}^{r+1}(0,T;\widetilde{\L}^{r+1}))\cap\mathcal{M}.$$
Thus, in \eqref{S24}, put $\phi=\mathrm{U}^n_{\k}-\varepsilon\tilde{\phi}\v$ where $\tilde{\phi}\in\mathrm{L}^{\infty}([0,T]\times\Omega;\d t\otimes\mathbb{P};\R)$ and $\v\in\V\cap\widetilde{\L}^{r+1}$, then divide it by $\varepsilon$ and passing limit $\varepsilon\to0$, we reach at
\begin{align}\label{S25}
	0&\geq\mathbb{E}\bigg[\int_{0}^{T}\psi(t)\bigg(\int_{0}^{t}e^{-\int_{0}^{s}\left(2\Phi(\phi(\zeta))+5\rho(\phi(\zeta))+\|\dot{\W}^n(\zeta)\|^2_{\mathrm{K}}+\|\k(\zeta)\|^2_{\mathrm{K}}\right)\d\zeta}\tilde{\phi}(s)\nonumber\\&\qquad\times\bigg\{2\big\langle\M_0^n(s)-\M(\phi(s)),\v\big\rangle+2\big\langle\mathrm{N}_{2,\k}^n(s)-\F_2(\phi(s))\dot{\W}^n(s)+\mathrm{N}_{3,\k}^n(s)\nonumber\\&\qquad-\F_3(\phi(s))\k(s)-\mathrm{N}_{\k}^n(s)+\F(\phi(s)),\v\big\rangle+\|\F_{1}(\phi(s))-\mathrm{N}^n_{1,k}(s)\|^2_{\mathcal{L}_2}\bigg\}\d s\bigg)\d t\bigg].
\end{align}
Since $\psi$ and $\tilde{\phi}$ are arbitrary, we finally obtain \eqref{S21} and \eqref{S22} from \eqref{S25}, and we conclude that $\mathrm{U}^n_{\k}$ are solutions of the system \ref{C_SCBF}. To obtain \eqref{E_U6}, we repeat the method used in the proof of Lemma \ref{5.1}, which completes the existence of solution of the system \eqref{C_SCBF}.

\vskip 2mm
\noindent
\textbf{Uniqueness:} For any $n\in\N$ given, let $\mathrm{U}^n_{1,\k}(\cdot), \mathrm{U}^n_{2,\k}(\cdot)$ be two solutions to the system \eqref{C_SCBF} with initial data $\x_1$ and $\x_2$ respectively. For $N>0$, we define 
\begin{align*}
	\tau_N^1&=\inf_{0\leq t\leq T}\Big\{t:\|\mathrm{U}^n_{1,\k}(t)\|_{\H}^2+\int_0^t\|\mathrm{U}^n_{1,\k}(s)\|_{\V}^2\d s\geq N\Big\},\\ \tau_N^2&=\inf_{0\leq t\leq T}\Big\{t:\|\mathrm{U}^n_{2,\k}(t)\|_{\H}^2+\int_0^t\|\mathrm{U}^n_{1,\k}(s)\|_{\V}^2\d s\geq N\Big\},\\ \tau_N:&=\tau_N^1\wedge\tau_N^2.
\end{align*}
 With the help of energy estimates \eqref{E_U6}, one can show that $\tau_N\to T$ as $N\to\infty$, $\mathbb{P}$-a.s. (see Proposition 3.5, \cite{MTM2}). Let us define $\mathcal{U}^n(\cdot):=\mathrm{U}^n_{1,\k}(\cdot)-\mathrm{U}^n_{2,\k}(\cdot)$. Then, $\mathcal{U}(\cdot)$ satisfies the following system:
\begin{equation}
	\left\{
	\begin{aligned}
		\d\mathcal{U}^n(t)&=-\left[\mu \A\mathcal{U}^n(t)+\alpha\mathcal{U}^n(t)+\B(\mathrm{U}^n_{1,\k}(t))-\B(\mathrm{U}^n_{2,\k}(t))+\beta(\mathcal{C}(\mathrm{U}^n_{1,\k}(t))-\mathcal{C}(\mathrm{U}^n_{2,\k}(t)))\right]\d t\\&\quad+\left[\mathrm{F}_1(\mathrm{U}^n_{1,\k})-\mathrm{F}_1(\mathrm{U}^n_{2,\k})\right]\d\W(t)+\left[\mathrm{F}_2(\mathrm{U}^n_{1,\k})-\mathrm{F}_2(\mathrm{U}^n_{2,\k})\right]\dot{\W}^n(t)\d t\\&\quad+\left[\mathrm{F}_3(\mathrm{U}^n_{1,\k})-\mathrm{F}_3(\mathrm{U}^n_{2,\k})\right]\k(t)\d t-\left[\mathrm{F}(\mathrm{U}^n_{1,\k})-\mathrm{F}(\mathrm{U}^n_{2,\k})\right]\d t,\\
		\mathcal{U}^n(0)&=\x_1-\x_2,
	\end{aligned}
	\right.
\end{equation}
in $\V'+\wi\L^{\frac{r+1}{r}}$. Let us take 
\begin{align*}
	\mathfrak{F}(t):=\int_{0}^{t}\left(2\Phi(\mathrm{U}^n_{2,\k}(\zeta))+5\rho(\mathrm{U}^n_{2,\k}(\zeta))+\|\dot{\W}^n(\zeta)\|^2_{\mathrm{K}}+\|\k(\zeta)\|^2_{\mathrm{K}}\right)\d\zeta,
\end{align*}
so that 
\begin{align*}
	\mathfrak{F}'(t)=2\Phi(\mathrm{U}^n_{2,\k}(t))+5\rho(\mathrm{U}^n_{2,\k}(t))+\|\dot{\W}^n(t)\|^2_{\mathrm{K}}+\|\k(t)\|^2_{\mathrm{K}}, \ \text{ for a. e. }\ t\in[0,T]. 
\end{align*}
Applying energy equality (see \cite{MTM2}, for a proof) to the process $e^{-\mathfrak{F}(t)}\|\mathcal{U}(t)\|^2_{\H}$ and using Theorem \ref{LocMon}, Hypotheses \ref{HonG} and \ref{HonG1}, we obtain 
\begin{align}\label{ITO}
	&e^{-\mathfrak{F}(\t)}\|\mathcal{U}^n(\s)\|_{\H}^2-\|\mathcal{U}^n(0)\|_{\H}^2\nonumber\\&= -2\int_0^{\t}e^{-\mathfrak{F}(s)}\langle\M(\mathrm{U}^n_{1,\k}(s))-\M(\mathrm{U}^n_{2,\k}(s)),\mathcal{U}^n(s)\rangle\d s\nonumber\\&\quad +2\int_0^{\t}e^{-\mathfrak{F}(s)}\left(\left[\mathrm{F}_2(\mathrm{U}^n_{1,\k}(s))-\mathrm{F}_2(\mathrm{U}^n_{2,\k}(s))\right]\dot{\W}^n(s),\mathcal{U}^n(s)\right)\d s\nonumber\\&\quad+2\int_0^{\t}e^{-\mathfrak{F}(s)}\left(\left[\mathrm{F}_3(\mathrm{U}^n_{1,\k}(s))-\mathrm{F}_3(\mathrm{U}^n_{2,\k}(s))\right]\k(s),\mathcal{U}^n(s)\right)\d s\nonumber\\&\quad+2\int_0^{\t}e^{-\mathfrak{F}(s)}\left(\mathrm{F}(\mathrm{U}^n_{2,\k}(s))-\mathrm{F}(\mathrm{U}^n_{1,\k}(s)),\mathcal{U}^n(s)\right)\d s\nonumber\\&\quad+\int_{0}^{\t}e^{-\mathfrak{F}(s)}\|\mathrm{F}_1(\mathrm{U}^n_{1,\k}(s))-\mathrm{F}_1(\mathrm{U}^n_{2,\k}(s))\|^2_{\mathcal{L}_2}\d s\nonumber\\&\quad+2\int_0^{\t}e^{-\mathfrak{F}(s)}\left(\left[\mathrm{F}_1(\mathrm{U}^n_{1,\k}(s))-\mathrm{F}_1(\mathrm{U}^n_{2,\k}(s))\right]\d\W(s),\mathcal{U}^n(s)\right)\d s-\int_{0}^{\t}\mathfrak{F}'(s)\|\mathcal{U}^n(s)\|^2_{\H}\d s\nonumber\\
	&\leq\int_0^{\t}e^{-\mathfrak{F}(s)}\mathfrak{F}'(s)\|\mathcal{U}^n(s)\|^2_{\H}\d s-\int_{0}^{\t}\mathfrak{F}'(s)\|\mathcal{U}^n(s)\|^2_{\H}\d s\nonumber\\&\quad+2\int_0^{\t}e^{-\mathfrak{F}(s)}\left(\left[\mathrm{F}_1(\mathrm{U}^n_{1,\k}(s))-\mathrm{F}_1(\mathrm{U}^n_{2,\k}(s))\right]\d\W(s),\mathcal{U}^n(s)\right)\d s.
\end{align}
Taking expectation in \eqref{ITO} and using the fact that the final term in \eqref{ITO} is a martingale, we obtain
\begin{align*}
&	\mathbb{E}\bigg[e^{-\mathfrak{F}(\t)}\|\mathrm{U}^n_{1,\k}(\t)-\mathrm{U}^n_{2,\k}(\t)\|^2_{\H}\bigg]\leq\mathbb{E}\bigg[\|\x_1-\x_2\|^2_{\H}\bigg].
\end{align*}
Thus the initial data  $\mathrm{U}^n_{1,\k}(0)=\mathrm{U}^n_{2,\k}(0)=\x$ yields to $\mathcal{U}^n(\t)=0$, $\mathbb{P}$-a.s. But we know that $\tau_N\to T$, $\mathbb{P}$-a.s. which gives $\mathcal{U}^n(t)=0$ and hence $\mathrm{U}^n_{1,\k}(t) = \mathrm{U}^n_{2,\k}(t)$, $\mathbb{P}$-a.s., for all $t \in[0, T ]$, which completes the proof of uniqueness.

	\medskip\noindent
{\bf Acknowledgments:}    The first author would like to thank the Council of Scientific $\&$ Industrial Research (CSIR), India for financial assistance (File No. 09/143(0938)/2019-EMR-I).  M. T. Mohan would  like to thank the Department of Science and Technology (DST), Govt of India for Innovation in Science Pursuit for Inspired Research (INSPIRE) Faculty Award (IFA17-MA110).

\end{document}